\documentclass[12pt]{article}
\usepackage{amsfonts}
\usepackage[all]{xy}
\newtheorem{theo}{Theorem}[section]
\newtheorem{prop}[theo]{Proposition}
\newtheorem{defi}{Definition}
\newtheorem{lem}[theo]{Lemma}

\newenvironment{proof}{{\bf Proof }}{\hfill $\Box$}

\newenvironment{rmq}{{\bf Remark }}{}

\newcommand{\C}{\mathbb{C}}
\newcommand{\N}{\mathbb{N}}
\newcommand{\R}{\mathbb{R}}

\newcommand{\Z}{\mathbb{Z}}
\newcommand{\Aa}{\mathcal{A}}

\newcommand{\Pp}{\mathcal{P}}
\newcommand{\Rr}{\mathcal{R}}

\begin{document}
\title{Multi-variable orthogonal polynomials }
\author{Abdallah Dhahri\\\vspace{-2mm}\scriptsize Department of Mathematics \\\vspace{-2mm}\scriptsize Faculty of Sciences of Tunis \\
\vspace{-2mm}\scriptsize University of Tunis El-Manar
\\\vspace{-2mm}\scriptsize1060 Tunis,Tunisia\\\vspace{-2mm}\scriptsize abdallah.dhahri@fst.rnu.tn\\ \\}

\date{}
\maketitle

\begin{abstract} We characterize the atomic probability measure on
${\R}^d$ which having a finite number of atoms. We further prove
that the Jacobi sequences associated to the multiple  Hermite (resp.
Laguerre, resp. Jacobi) orthogonal polynomials are diagonal
matrices. Finally, as a consequence of the multiple Jacobi
orthogonal polynomials case, we give the Jacobi sequences of the
 Gegenbauer, Chebyshev and  Legendre  orthogonal polynomials.
\end{abstract}
\section{Introduction} Let $\mu$ be a probability measure on ${\R}$
with finite moments of all orders. Apply the Gram-Schmidt
orthogonalization process to the sequence
$\left\{1,x^2,\dots,x^n,\dots\right\}$ to get a sequence
$\left\{P_n(x); \ \ n=0,1,\dots \right\}$ of orthogonal polynomials
in $L^2(\mu)$, where $P_0(x)=1$ and $P_n(x)$ is a polynomial of
degree $n$ with leading coefficient 1. It is well-known that these
polynomials $P_n$'s satisfy the recursion formula:
$$(x-\alpha_n)P_n(x)=P_{n+1}(x)+\omega_nP_{n-1}(x), \ \ n\geq0.$$
where $\alpha_n\in{\R},\omega_n\geq0$ for $n\geq0$ and $P_{-1}=0$ by
convention. The sequences $(\alpha_n)_n$ and $(w_n)_n$ are called
the Jacobi sequences associated to the probability measure $\mu$ (cf
\cite{ch1978}, \cite{Is-As}, \cite{sz1975}).\\

 In the multi-dimensional case (cf \cite{DX},\cite{K1}, \cite{K2},\cite{X})
the formulations of these results are recently given by identifying
the theory of multi-dimensional orthogonal polynomials with the
theory of symmetric interacting Fock spaces (cf \cite{AcBaDh}).
 The multi-dimensional analogue of positive numbers $w_n$ (resp. real numbers $\alpha_n$)
 are the positive definite matrices (resp. Hermitean matrices).

In this paper, we characterize the atomic probability measures on
${\R}^d$ which having a finite number of atoms. Moreover, we give
the Jacobi sequences associated to the multiple Hermite (resp.
Laguerre, resp. Jacobi) orthogonal polynomials and we prove that
they are diagonal matrices. As a corollary of the Jacobi case, we
give the explicit forms of the ones associated to the Gegenbaur,
Chebyshev  and Legendre orthogonal polynomials.

This paper is organized as follows. In section 2 we recall the basic
properties of the complex polynomial algebra in $d$ commuting
indeterminates  and we give the multi-dimensional Favard Lemma. The
characterization of the atomic probability  measures on ${\R}^{d}$
which having a finite number of atoms is given in section 3.
Finally, in section 4 we give the explicit forms of the Jacobi
sequences associated to the multiple Hermite (resp. Laguerre, resp.
Jacobi) orthogonal polynomials.

%%%%%%%%%%%%%%%%%%%%%%%%%%%%%%%%%%%%%%%%%%%%%%%%%%%%%%%%%%%%%%%%%%%%%%%%%%%%%%%%%%%%%%%%%%%%%%%%%%%%%%%%%%%%%%%%%%

\section{The multi-dimensional Favard Lemma}
In this section we recall the basic properties of the polynomial
algebra in $d$ commuting indeterminates and we give the
multi-dimensional Favard Lemma. We refer the interested reader to
\cite{AcBaDh} for more details.

\subsection{The polynomial algebra in $d$ commuting indeterminates}
Let $d\in \mathbb{N}^*$ and let
$$\mathcal{P}=\mathbb{C}[(X_j)_{1\leq j\leq d}]$$
be the complex polynomial algebra in the commuting indeterminates
$(X_j)_{1\leq j\leq d}$ with the $*$-structure uniquely determined
by the prescription that the $X_j$ are self-adjoint. For all
$v=(v_1,\dots,v_d)\in\mathbb{C}^d$ denote
$$X_v:=\sum_{j=1}^dv_jX_j$$

A {\it monomial} of degree $n\in\mathbb{N}$ is by definition any
product of the form
$$M:=\prod_{j=1}^d X_j^{n_j}$$
where, for any $1\leq j\leq d$, $n_j\in\mathbb{N}$ and
$n_1+\dots+n_d=n$.

Denote by $\mathcal{P}_{n]}$ the vector subspace of $\mathcal{P}$
generated by the set of monomials of degree less or equal than $n$.
It is clear that
$$\mathcal{P}=\cup_{n\in\mathbb{N}}\mathcal{P}_{n]}$$
\begin{defi}
For $n\in\mathbb{N}$ we say that a subspace
$\mathcal{P}_n\subset\mathcal{P}_{n]}$ is monic of degree $n$ if
$$\mathcal{P}_{n]}=\mathcal{P}_{n-1]}\dot+\mathcal{P}_n$$
(with the convention $\mathcal{P}_{-1]}=\{0\}$ and where $\dot+$
means a vector space direct sum) and $\mathcal{P}_n$ has a linear
basis $\mathcal{B}_n$ with the property that for each
$b\in\mathcal{B}_n$, the highest order term of $b$ is a non-zero
multiple of a monomial of degree $n$. Such a basis is called a
perturbation of the monomial basis of order $n$ in the coordinates
$(X_j)_{1\leq j\leq d}$.
\end{defi}

Note that any state $\varphi$ on $\mathcal{P}$ defines a pre-scalar
product
\begin{eqnarray*}
\langle.,.\rangle_\varphi:\mathcal{P}\times\mathcal{P}&\rightarrow& \mathbb{C}\\
(a,b)&\mapsto&\langle a,b\rangle_\varphi=\varphi(a^*b)
\end{eqnarray*}
with $\langle1_\mathcal{P},1_\mathcal{P}\rangle_\varphi=1$.
\begin{lem}\label{state}
Let $\varphi$ be a state on $\mathcal{P}$ and denote $\langle\
\cdot,\cdot \ \rangle=\langle\ \cdot,\cdot \ \rangle_\varphi$ be the
associated pre-scalar product. Then there exists a gradation
\begin{equation}\label{phi-gradation}
\mathcal{P}=\bigoplus_{n\in\mathbb{N}}\left(\mathcal{P}_{n,\varphi},\:\langle
\ \cdot,\cdot \ \rangle_{n,\varphi}\right)
\end{equation}
called a $\varphi$-orthogonal polynomial decomposition of
$\mathcal{P}$, with the following properties:
\begin{enumerate}
\item[(i)] (\ref{phi-gradation}) is orthogonal for the unique pre-scalar product $\langle\ \cdot,\cdot\ \rangle$ on $\mathcal{P}$ defined by the conditions:
\begin{eqnarray*}
\langle\ \cdot,\cdot\ \rangle|_{\mathcal{P}_n,\varphi}&=&\langle \ \cdot,\cdot \ \rangle_{n,\varphi},\qquad \forall n\in\mathbb{N}\\
\mathcal{P}_{m,\varphi}&\bot&\mathcal{P}_{n,\varphi}, \qquad \forall
m\neq n
\end{eqnarray*}
\item[(ii)] (\ref{phi-gradation}) is compatible with the filtration $(\mathcal{P}_{n]})_n$ in the sense that
$$\mathcal{P}_{n]}=\bigoplus_{h=0}^n\mathcal{P}_{h,\varphi},\qquad \forall n\in\mathbb{N},$$
\item[(iii)] for each $n\in\mathbb{N}$ the space $\mathcal{P}_{n,\varphi}$ is monic.
\end{enumerate}

Conversely, let be given:
\begin{enumerate}
\item[(j)] a vector space direct sum decomposition of $\mathcal{P}$
\begin{equation}\label{direct-sum}
\mathcal{P}=\sum_{n\in\mathbb{N}}^{\cdot}\mathcal{P}_n
\end{equation}
such that $\mathcal{P}_0=\mathbb{C}.1_\mathcal{P}$, and for each
$n\in\mathbb{N}$, $\mathcal{P}_n$ is monic of degree $n$,
\item[(jj)] for all $n\in\mathbb{N}$ a pre-scalar product $\langle \ \cdot,\cdot \ \rangle_{n}$ on $\mathcal{P}_n$ with the property that $1_\mathcal{P}$ has norm $1$ and the unique pre-scalar product $\langle \ \cdot,\cdot \ \rangle$ on $\mathcal{P}$ defined by the conditions:
\begin{eqnarray*}
\langle\ \cdot,\cdot\ \rangle|_{\mathcal{P}_n}&=&\langle \ \cdot,\cdot \ \rangle_{n},\qquad \forall n\in\mathbb{N}\\
\mathcal{P}_{m}&\bot&\mathcal{P}_{n}, \qquad \forall m\neq n
\end{eqnarray*}
satisfies $\langle 1_\mathcal{P},1_\mathcal{P}\rangle=1$ and
multiplication by the coordinates $X_j$ ($1\leq j\leq d$) are
$\langle\ \cdot,\cdot\ \rangle$-symmetric linear operators on
$\mathcal{P}$.\smallskip
\end{enumerate}
Then there exists a state $\varphi$ on $\mathcal{P}$ such that the
decomposition (\ref{direct-sum}) is the orthogonal polynomial
decomposition of $\mathcal{P}$ with respect to $\varphi$.
\end{lem}

\subsection{The symmetric Jacobi relations and the CAP operators}
In the following we fix a state $\varphi$ on $\mathcal{P}$ and we
follow the notations of Lemma \ref{state} with the exception that we
omit the index $\varphi$. We write $\langle.,.\rangle$ for the
pre-scalar product $\langle.,.\rangle_\varphi$, $\mathcal{P}_k$ for
the space $\mathcal{P}_{k,\varphi}$ and
$P_{k]}:\mathcal{P}\rightarrow \mathcal{P}_{k]}$ the
$\langle.,.\rangle$-orthogonal projector in the pre-Hilbert space
sense (see \cite{AcBaDh} for more details). Put
$$P_n=P_{n]}-P_{n-1]}$$
It is obvious that $P_n=P_n^*$ and $P_nP_m=\delta_{nm}P_n$ for all
$n,m\in\mathbb{N}$.

It is proved in \cite{AcBaDh} that for any $1\leq j\leq d$ and any
$n\in\mathbb{N}$, one has
\begin{eqnarray}\label{Jacobi relation}
X_jP_n=P_{n+1}X_jP_n+P_nX_jP_n+P_{n-1}X_jP_n
\end{eqnarray}
with the convention that $P_{-1]}=0$. The identity (\ref{Jacobi
relation}) is called the {\it symmetric Jacobi relation}.

Now for each $1\leq j\leq d$ and $n\in\mathbb{N}$ we define the
operators $a^\varepsilon_{j|n},$ $\varepsilon\in\{+,0,-\}$, with
respect to a basis $e=(e_j)_{1\leq j\leq d}$ of $\mathbb{C}^d$ as
follows:
\begin{eqnarray}
a^{+}_{j|n}&=&a^+_{e_j|n}:=P_{n+1}X_jP_{n}\Big|_{{{\cal P}}_{n}}
\ : \ {{\cal P}}_{n}\longrightarrow {{\cal P}}_{n+1}\label{1df-a+j|n}\nonumber\\
a^{0}_{j|n}&=&a^0_{e_j|n}:=P_{n}X_jP_{n}\Big|_{{{\cal P}}_{n}}
\ : \ {{\cal P}}_{n}\longrightarrow {{\cal P}}_{n}\label{cre-ann-con}\\
a^{-}_{j|n}&=&a^-_{e_j|n}:=P_{n-1}X_jP_{n}\Big|_{{{\cal P}}_{n}} \ :
\ {{\cal P}}_{n}\longrightarrow {\cal
P}_{n-1}\label{1df-a-j|n}\nonumber
\end{eqnarray}

{\bf Notation}: If $v=(v_1,\dots,v_d)\in\mathbb{C}^d$, where
$v_1,\dots,v_d$ are the coordinates of $v$ in the basis $e$, we
denote
\begin{eqnarray*}\label{notation}
a^{\varepsilon}_{v|n}:=\sum_{1\leq j\leq d}v_ja^{\varepsilon}_{j|n}
\end{eqnarray*}

Note that in this context, the sum
\begin{equation}\label{weak-sum}
{{\cal P}}=\bigoplus_{n\in{\mathbb N}}{{\cal P}}_{n}
\end{equation}
is orthogonal and meant in the weak sense, i.e. for each element
$Q\in {{\cal P}}$ there is a finite set $I\subset{\mathbb N}$ such
that
\begin{equation}\label{finite-decomposition}
    Q=\sum_{n\in I}p_n
\; ,\qquad p_n\in {{\cal P}}_{n}
\end{equation}

\begin{theo}\label{th-Q--dec}
On ${{\cal P}}$, for any $1\leq j\leq d $, the following operators
are well defined
\begin{eqnarray*}\label{1df-a+j}
a^{+}_j&:=&\sum_{n\in{\mathbb N}}a^{+}_{j|n}\\
a^{0}_j&:=&\sum_{n\in{\mathbb N}}a^{0}_{j|n}\\
a^{-}_j&:=&\sum_{n\in{\mathbb N}}a^{-}_{j|n}
\end{eqnarray*}
and one has
\begin{equation}\label{q-dec-Xj}
X_j=a^{+}_j+a^{0}_j+a^{-}_j
\end{equation}
in the sense that both sides of (\ref{q-dec-Xj}) are well defined on
${{\cal P}}$ and the equality holds.
\end{theo}
Identity (\ref{q-dec-Xj}) is called a {\it quantum decomposition} of
the variable $X_j$.
\begin{prop}
For any $1\leq j\leq d$ and $n\in\mathbb N$, one has
\begin{eqnarray*}
(a^{+}_{j|n})^{*} &=& a^{-}_{j|n+1}\qquad; \qquad
(a^{+}_{j})^{*} = a^{-}_{j}\label{a+j*=a-j}\\
(a^{0}_{j|n})^{*} &=& a^{0}_{j|n}\qquad; \qquad
(a^0_{j})^*=a^0_{j}\label{a0j*=a0-j}
\end{eqnarray*}
Moreover, for each $j,k\in\{1,\dots,d\}$, one has
$$[a^+_j,a^+_k]=0$$
\end{prop}
\subsection{$3$-diagonal decompositions of $\mathcal{P}$ and multi-dimensional Favard Lemma}

\begin{defi}\label{df-3d-dec}
For $n\in\mathbb N$ a {\it $3$--diagonal decomposition of} $\mathcal
P_{n]}$
$$
\left\{\left({\cal P}_k \ , \ \langle\cdot,\cdot\rangle_{k}
\right)^{n}_{k=0} \ , \ \left(a^{+}_{\cdot|k}\right)^{n-1}_{k=0} \ ,
\ \left(a^{0}_{\cdot|k}\right)^{n}_{k=0}\right\}
$$
is defined by:
\begin{enumerate}
\item[(i)] a vector space direct sum decomposition of
$\mathcal P_{n]}$ such that
\begin{equation}\label{orth-dec-Pk]}
\mathcal P_{k]} = \sum_{h\in\{0,\dots,k\}}^{\cdot} \mathcal P_{h}
\qquad;\qquad \forall k\in\{0,1,\cdots , n\}
\end{equation}
where each $\mathcal{P}_k$ is monic.
\item[(ii)] for each $k\in\{0,1,\cdots , n\}$ a pre-scalar product
$\langle \ \cdot \ , \ \cdot \ \rangle_{k} $ on $\mathcal P_{k}$.
\item[(iii)] two families of linear maps
\begin{eqnarray*}
v\in \mathbb C^d &\longmapsto&a^+_{v|k}\in\mathcal L(\mathcal P_{k}  ,  \mathcal P_{k+1})\qquad ,\qquad k\in\{0,1,\cdots , n-1\}\label{a+v|k}\\
v\in \mathbb C^d &\longmapsto&a^0_{v|k}\in\mathcal L(\mathcal P_{k},
\mathcal P_{k}) \qquad ,\qquad k\in\{0,1,\cdots , n\}
\end{eqnarray*}
\end{enumerate}
such that:
\begin{enumerate}
\item[-] for all $v\in \mathbb R^d $, $a^+_{v|k}$ maps the $(\mathcal{P}_k,\langle.,.\rangle_{k})$-zero norm subspace into the $(\mathcal{P}_{k+1},\langle.,.\rangle_{k+1})$-zero norm subspace;
\item[-] for all $v\in \mathbb R^d $, $a^0_{v|k}$ is a
self-adjoint operator on the pre-Hilbert space $(\mathcal P_{k} \ ,
\ \langle \ \cdot  ,  \cdot \ \rangle_{k})$, thus in particular it
maps $(\mathcal{P}_k,\langle\ \cdot,\cdot\ \rangle_{k})$-zero norm
subspace into itself;
\item[-] denoting $*$ (when no confusion is possible)
the adjoint of a linear map from $\left(\mathcal P_{k-1} \ , \
\langle \ \cdot , \cdot \ \rangle_{k-1}\right)$ to  $\left(\mathcal
P_{k} \ , \ \langle \ \cdot  ,  \cdot \ \rangle_{k}\right)$ for any
$k\in\{0,1,\cdots , n\}$, and defining
\begin{eqnarray*}\label{df-a-v|k}
a^-_{v|k} := (a^+_{v|k-1})^*  \;; \qquad a^+_{v|-1} :=0 \;;\qquad
k\in\{0,1,\cdots , n-1\} \ , \ v\in \mathbb C^d
\end{eqnarray*}
the following identity is satisfied:
\begin{eqnarray*}\label{df-a+v|k1}
X_v\Big|_{\mathcal P_{k}} = a^+_{v|k} + a^{0}_{v|k} + a^-_{v|k}
\qquad; \qquad k\in\{0,1,\cdots , n-1\} \ , \ v\in\mathbb R^d
\end{eqnarray*}
\end{enumerate}
\end{defi}

{\bf Remarks}: For the following remarks we refer to \cite{AcBaDh}.
\begin{enumerate}
\item[(i)] Any 3-diagonal decomposition of $\mathcal{P}_{n]}$ induces, by restriction, a 3-diagonal decomposition of $\mathcal{P}_{k]}$ for any $k\leq n$.
\item[(ii)] By definition
$$\mathcal{P}_n:=\{a^+_{v|n}(\mathcal{P}_{n-1});\:\:v\in\mathbb{C}^d\}$$
\end{enumerate}
\begin{theo}
The $3$-diagonal decompositions of ${\mathcal P}$ are in one-to-one
correspondence with the pre-scalar products on $\mathcal{P}$ induced
by some state $\varphi$ on $\mathcal{P}$.
\end{theo}

In the following $\otimes$ will denote the algebraic tensor product
and $\hat{\otimes}$ its symmetrization. The tensor algebra over
$\mathbb{C}^d$ is the vector space
$$\mathcal{T}(\mathbb{C}^d):=\sum^{\cdot}_{n\in\mathbb{N}}(\mathbb{C}^d)^{\otimes n}$$
with multiplication given by
$$(u_n\otimes\dots\otimes u_1)\otimes(v_n\otimes\dots\otimes v_1):=u_n\otimes\dots\otimes u_1\otimes v_n\otimes\dots\otimes v_1$$
for all $n,m\in\mathbb{N}$ and all $u_j,v_j\in\mathbb{C}^d$. The
$*$-sub-algebra of $\mathcal{T}(\mathbb{C}^d)$ generated by the
elements of the form
$$v^{\otimes n}:=v\otimes\dots\otimes v\:(n-times),\:\forall n\in\mathbb{N},\;\forall v\in\mathbb{C}^d$$
is called the {\it symmetric tensor algebra} over $\mathbb{C}^d$ and
denoted $\mathcal{T}_{sym}(\mathbb{C}^d)$.

\begin{lem}\label{id-symm-tens}
For all $n\in\mathbb{N}^*$, let ${\cal P}_{n}$ be the $n-th$ space
of a $3$-diagonal decomposition of $\mathcal{P}$. Denoting, for
$v\in\mathbb{C}^d$, $a^+_v:=\sum_{n\in\mathbb{N}}a^+_{v|k}$ and
$\Phi=1_\mathcal{P}$. Then the map
\begin{equation}\label{symm-tens-vn-a+vn}
 U_n:\;v_n\hat\otimes v_{n-1}\hat\otimes\cdots \hat\otimes v_{1}
 \in (\mathbb C^d)^{\hat \otimes n} \ \longmapsto \
a^+_{v_n}a^+_{v_{n-1}} \cdots a^+_{v_{1}}\Phi \in {\cal P}_{n},
\end{equation}
extends uniquely to a vector space isomorphism with the property
that for all $v\in\mathbb{C}^d$ and
$\xi_{n-1}\in(\mathbb{C}^d)^{\hat \otimes (n-1)}$
$$U_n(v\hat\otimes\xi_{n-1})=a^+_{v}U_{n-1}\xi_{n-1}$$

For $n=0$ we put
$$U_0:z\in\mathbb{C}:=(\mathbb{C}^d)^{\hat\otimes0}\longmapsto U_0(z):=z\in\mathbb{C}1_\mathcal{P}\in\mathcal{P}_0$$
\end{lem}

The multi-dimensional Favard Lemma is given by the following
theorem.
\begin{theo}\label{multi}
Let $\mu$ be a probability measure on ${\mathbb R}^d$ with finite
moments of all orders and denote $\varphi$ the state on
$\mathcal{P}$ given by
$$\varphi(b)=\int_{\mathbb{R}^d}b(x_1,\dots,x_d)d\mu(x_1,\dots,x_d),\:b\in\mathcal{P}$$
Then there exist two sequences
$$
(\Omega_n)_{n\in{\mathbb N}}\qquad ;\qquad
(\alpha_{.|n})_{n\in{\mathbb N}}
$$
satisfying:
\begin{enumerate}
\item[(i)] for all $n\in{\mathbb N}, \Omega_n$ is a linear operator on $(\mathbb{C}^d)^{\hat\otimes n}$ positive and symmetric with respect to the tensor scalar product given by
$$
\langle u^{\otimes n},v^{\otimes
m}\rangle_{(\mathbb{C}^d)^{\hat{\otimes}n}}:=\delta_{m,n}\langle
u,v\rangle_{\mathbb{C}^d}^n;\:\: \forall u,v\in \mathbb{C}^d;
\forall n\in{\mathbb N}
$$
where $\langle\ \cdot,\cdot\ \rangle_{\mathbb{C}^d}$ is a pre-scalar
product on $\mathbb{C}^d$.
\item[(ii)] denoting for all $n\in\mathbb{N}$
\begin{equation}\label{6.1}
\langle\xi_n,\eta_n\rangle_n:=\langle\xi_n,\Omega_n\eta_n\rangle_{({\mathbb
C}^d)^{\hat{\otimes}n}};\:\:\:\xi_n,\eta_n\in(\mathbb{C}^d)^{\hat\otimes
n}
\end{equation}
the pre-scalar product on $(\mathbb{C}^d)^{\hat{\otimes}n}$ defined
by $\Omega_n$ and $| \ \cdot,\cdot \ |_n$ the associated pre-norm.
For all $n\in\mathbb{N},v\in\mathbb{C}^d$ and
$\eta_{n-1}\in(\mathbb{C}^d)^{\hat{\otimes}(n-1)}$, one has
\begin{equation}\label{*0}
|\eta_{n-1}|_{n-1}=0\Rightarrow |v\hat\otimes \eta_{n-1}|_n=0
\end{equation}
\item[(iii)] for all $n\in{\mathbb N},$
$$\alpha_{.|n} \ : \ v\in{\mathbb C}^d \ \to \ \alpha_{v|n}\in\mathcal{L}\Big((\mathbb{C}^d)^{\hat{\otimes}n}\Big)
$$
is a linear map and for all $v\in{\mathbb R}^d$, $\alpha_{v|n}$ is a
linear operator on $({\mathbb C}^d)^{\hat{\otimes}n}$, symmetric for
the pre-scalar product $\langle \ \cdot  ,  \cdot \ \rangle_n$ on
$({\mathbb C}^d)^{\hat{\otimes}n}$;
\item[(iv)] the sequence $\Omega_n$ defines a symmetric interacting Fock space struture over $\mathbb{C}^d$ endowed with the tensor pre-scalar product (\ref{6.1}) and the operator
\begin{equation}\label{gradation-preserving}
\!\!\!\!\!\!\!\!\!\!\!\!\!U:=\bigoplus_{k\in\mathbb{N}}U_k:\bigoplus_{k\in\mathbb{N}}\left(
({\mathbb C}^d)^{\hat{\otimes}k} ,  \ \langle  \cdot  , \cdot
\rangle_k\right)\rightarrow\bigoplus_{k\in\mathbb{N}}\left(\mathcal{P}_k
,  \ \langle  \cdot  , \cdot
\rangle_{\mathcal{P}_k}\right)=\left(\mathcal{P},\ \langle  \cdot  ,
\cdot \rangle\right)
\end{equation}
is an orthogonal gradation preserving unitary isomorphism of
pre-Hilbert spaces, where $\langle  \cdot  , \cdot
\rangle_{\mathcal{P}_k}$ is the pre-scalar product induced by
$\varphi$ on $\mathcal{P}_k$.
\end{enumerate}
Moreover, denoting
\begin{eqnarray}\label{FAv-IFS}
\Gamma\left({\mathbb C}^d,\;(\Omega_n)_n\right)
:=\bigoplus_{n\in{\mathbb N}}\left( ({\mathbb C}^d)^{\hat{\otimes}n}
\ ,  \ \langle \ \cdot  ,  \cdot \ \rangle_n\right)
\end{eqnarray}
the symmetric interacting Fock space defined by the sequence
$(\Omega_n)_{n\in{\mathbb N}}$, $A^\pm$ the creation and
annihilation fields associated to it, $P_{\Gamma,n}$ the projection
onto the $n-th$ space of the gradation (\ref{FAv-IFS}), and $N$ the
number operator associated to this gradation i.e.
$$N:=\sum_{n\in\mathbb{N}}nP_{\Gamma,n},$$
the gradation preserving unitary pre-Hilbert space isomorphism
(\ref{gradation-preserving}) satisfies
\begin{eqnarray*}
U\Phi&=&1_\mathcal{P}\\
U^{-1}X_vU&=&A^{+}_v+\alpha_{v,N}+A^{-}_v,\:\:\forall
v\in\mathbb{R}^d,
\end{eqnarray*}
where $\alpha_{v,N}$ is the symmetric operator defined by
$$\alpha_{v,N}:=\sum_{n\in\mathbb{N}}\alpha_{v|n}P_{\Gamma,n}.$$

Conversely, given two sequences $(\Omega_n)_{n\in{\mathbb N}}$ and
$(\alpha_{.|n})_{n\in{\mathbb N}}$ satisfying (i), (ii), (iii) and
(iv) above, there exists a state $\varphi$ on $\mathcal{P}$, such
that for any probability measure $\mu$ on ${\mathbb R}^d$, inducing
the state $\varphi$ on $\mathcal{P}$, the pair of
sequences\\$\left((\Omega_n)_{n\in\mathbb{N}},\:(\alpha_{.|n})_{n\in\mathbb{N}}\right)$
is the one associated to $\mu$ according to the first part of the
theorem.
\end{theo}

{\bf Remark}:
\begin{enumerate}
\item[1)] From the proof of the above theorem (cf \cite{AcBaDh}) one has
\begin{equation}\label{alphan}
\alpha_{.|n}=U_n^{-1}a^0_{.|n}U_n
\end{equation}
\item[2)] from (\ref{*0}), it follows that if there exists $n_0\in{\N}^*$ such that $\Omega_{n_0}=0$, then $$\Omega_n=0, \ \ \forall n\geq n_0.$$
\end{enumerate}
\begin{defi}
The sequences $(\Omega_n)_n$ and $(\alpha_{.|n})_n$ in Theorem
\ref{multi} are called Jacobi sequences associated to the
probability measure $\mu$.
\end{defi}

%%%%%%%%%%%%%%%%%%%%%%%%%%%%%%%%%%%%%%%%%%%%%%%%%%%%%%%%%%%%%%%%%%%%%%%%%%%%%%%%%%%%%%%%%%%%%%%%%%%%%%%%%%%%%%%%%%%%%%%

\section{Positive Jacobi sequence and atomic probability measure}

Recall that for each $n\in{\N}$ the positive matrix
$\Omega_n\in\mathcal{L}(({\C}^d)^{\widehat{\otimes}n})$.
\begin{prop} If there exists $n_0\in{\N}$ such that $rank(\Omega_{n_0})<dim\Big(({\C^d})^{\widehat{\otimes}n_0}\Big)$, then for all
$k\in{\N},rank(\Omega_{n_0+k})<dim\Big(({\C^d})^{\widehat{\otimes}n_0+k}\Big)$.
\end{prop}

\begin{proof} Suppose that there exists $n_0\in{\N}$ such that
 rank $(\Omega_{n_0})<dim\Big(({\C^d})^{\widehat{\otimes}n_0}\Big)$
i.e. $\Omega_{n_0}$ is not injective. Let
$\xi_{n_0}\in({\C}^d)^{\widehat{\otimes}n_0},\xi_{n_0}\neq0_{({\C}^d)^{\widehat{\otimes}n_0}}$
such that $\Omega_{n_0}(\xi_{n_0})=0$, then for all $v\in{\C}^d$,
for all arbitrary $\eta_{n_0+1}\in({\C}^d)^{\widehat{\otimes}n_0+1}$
one has
\begin{eqnarray*}
\langle\eta_{n_0+1},\Omega_{n_0+1}(v\widehat{\otimes}\xi_{n_0})\rangle_{({\C^d})^{\widehat{\otimes}n_0+1}}
&=&\langle U_{n_0+1}(\eta_{n_0+1}),U_{n_0+1}(v\widehat{\otimes}\xi_{n_0})\rangle_{{\Pp}_{n_0+1}}\\
&=&\langle U_{n_0+1}(\eta_{n_0+1}),a^+_v U_{n_0}(\xi_{n_0})\rangle_{{\Pp}_{n_0+1}}\\
&=&\langle a^-_v U_{n_0+1}(\eta_{n_0+1}), U_{n_0}(\xi_{n_0})\rangle_{{\Pp}_{n_0}}\\
&\leq& |a^-_v
U_{n_0+1}(\eta_{n_0+1})|_{{\Pp}_{n_0}}|U_{n_0}(\xi_{n_0})|_{{\Pp}_{n_0}}
\end{eqnarray*}
Because $\Omega_{n_0}(\xi_{n_0})=0$ i.e. $\langle
\xi_{n_0},\Omega_{n_0}(\xi_{n_0})\rangle_{({\C}^d)\widehat{\otimes}n_0}=|U_{n_0}(\xi_{n_0})|^2_{{\Pp}_{n_0}}=0$,
one has
$$\Omega_{n_0+1}(v\widehat{\otimes}\xi_{n_0})=0, \ \ \forall
v\in{\C}^d$$ It follows by induction on $k\in{\N}$
$$\Omega_{n_0+k}(v^{\widehat{\otimes}k}\widehat{\otimes}\xi_{n_0})=0, \ \ \forall
v\in{\C^d}, \forall k\in{\N}$$ If $v\neq0_{{\C}^d}$, because
$\xi_{n_0}\neq0_{({\C}^d)^{\widehat{\otimes}n_0}}$, one gets
$$v^{\widehat{\otimes}k}\widehat{\otimes}\xi_{n_0}\subset Ker(\Omega_{n_0+k}), \ \ \forall
k\in{\N}$$ and
$$v^{\widehat{\otimes}k}\widehat{\otimes}\xi_{n_0}\neq0_{({\C}^d)^{\widehat{\otimes}n_0+k}}$$
Hence, $\Omega_{n_0+k} \ \ (k\in{\N})$ is not injective i.e.
$rank(\Omega_{n_0+k})<dim\Big(({\C^d})^{\widehat{\otimes}n_0+k}\Big)$.
\end{proof}

Now, our aim is to give a characterization of the atomic probability
measure on ${\R}^d$ which have a finite number of atoms.

A common zero of a set of polynomials is a zero for every polynomial
in the set. Let $\mu$ be a probability measure on ${\R}^{d}.$ Let
$\mathbb{P}_{n}=\left\{P_{\alpha}^{n}\right\}_{\alpha}$ be a
sequence of orthogonal polynomials with respect to $\mu$, where
$\alpha =(\alpha_{1},\alpha_{2},\dots,\alpha_{d})\in{\N}^{d}$ and
$|\alpha|=\alpha_{1}+\alpha_{2}+\dots+\alpha_{d}=n$. A common zero
of $\mathbb{P}_{n}$ is a zeros of every $P_{\alpha}^{n}.$ Clearly we
can consider zeros of $\mathbb{P}_{n}$ as  zeros of the subspace
${\Pp}_{n}.$ For the following lemma we refer the reader to
\cite{G.F.D.}.

\begin{lem}\label{l1} The polynomials in  $\mathbb{P}_{n}$ have at
most $dim{\Pp}_{n-1]}$ common zeros.
\end{lem}

\begin{defi} Given a measurable space $(X,\Sigma)$ and a measure $\mu$ on that space, a set A in $\Sigma$ is called an atom if
$\mu(A)>0$ and for any measurable subset B of A with
$\mu(A)>\mu(B)$, one has $\mu(B)=0$.
\end{defi}
\begin{defi} Given a measurable space $(X,\Sigma)$ and a measure $\mu$ on that
space. $\mu$ is said an atomic if there is a partition of $X$ into
countably many elements of $\Sigma$ which are either atoms or null
sets.
\end{defi}
\begin{rmq} If $\mu$ is a $\sigma$-finite probability measure on the Borel $\sigma-$algebra of
${\R}^{n}$, then it is easy to show that, for any atom $B$ of $\mu$
there is a point $x\in{B}$ with the property that
$\mu(B)=\mu(\left\{x\right\})$. Thus such a measure is atomic if and
only if it is the countable sum of Dirac deltas, i.e. if there is an
(at most) countable set $\left\{x_{i}\right\}\subset{\R}^{n}$ and an
(at most) countable set $\left\{a_{i}\right\}\subset]0,\infty[$ with
the property that
$$\mu(A)=\sum_{x_{i}\in{A}} a_{i} ~~for ~every ~Borel ~set ~A.$$
i.e. $\mu=\sum_{i}a_{i}\delta_{x_{i}},$ with $\sum_{i}a_{i}=1.$
\end{rmq}

\begin{theo}
There exists $n_{0}\in {\N^{*}}$ such that the matrices $\Omega_{n}$
are zero for $n\geq n_{0}$ of and only if the associated probability
measure is atomic and having a finite number of atoms.
\end{theo}
\begin{proof}
Let $\mu$ a probability measure with finite moments of any order and
suppose that there exists $n_{0}\in {\N^{*}}$ such that the matrices
$\Omega_{n_{0}}=0$. It follows that for all
$\xi_{n_{0}}\in({\C}^{d})^{\widehat{\otimes} n_{0}}$
$$0=\langle\xi_{n_{0}},\Omega_{n_{0}}\xi_{n_{0}}\rangle_{({\C}^{d})^{\widehat{\otimes} n_{0}}}=
\langle U_{n_{0}}\xi_{n_{0}},U_{n_{0}}\xi_{n_{0}}\rangle_{\mu}.$$
Since $U_{n_{0}}\in Isom(({\C}^{d})^{\widehat{\otimes}
n_{0}},{\Pp}_{n_{0}})$ and $\Omega_{n_{0}}=0$, then one has
$$\langle Q_{1},Q_{2}\rangle_{\mu}=0,\:\: \forall
Q_{1},Q_{2}\in{\Pp}_{n_{0}}$$
 It follows that
$$\int_{{\R}^{d}}|Q(x)|^{2}\mu(dx)=0,~~\forall Q\in{\Pp}_{n_{0}}.$$
Thus for all $Q\in {\Pp}_{n_{0}}$, one has
$$Q=0~\mu.a.s.~~i.e.~~\mu(\left\{x\in{\R}^{d};~~Q(x)=0\right\})=1.$$
Let
$\mathbb{P}_{n_{0}}=\left\{P_{\alpha}^{n_{0}}\right\}_{|\alpha|=n_0}$
be an orthogonal basis of ${\Pp}_{n_{0}}$. Put
$$\Delta_{\alpha}=\left\{x\in{\R}^{d};~~P_{\alpha}^{n_{0}}(x)=0\right\}$$
and $$\mathcal{D}_{n_{0}}=\cap_{|\alpha|=n_{0}}\Delta_{\alpha}.$$ It
is clear that for any $\alpha$ such that $|\alpha|=n_{0}$,
$\mu(\Delta_\alpha)=1$. Moreover, one has
$\mu(\mathcal{D}_{n_{0}})=1$ because
$$\mu(\mathcal{D}_{n_{0}}^{c})=\mu(\cup_{|\alpha|=n_{0}}\Delta_{\alpha}^{c})\leq\sum_{|\alpha|=n_{0}}\mu(\Delta_{\alpha}^{c})=0.$$
Thus, one gets
$$\mathcal{D}_{n_{0}}\neq\emptyset.$$
Moreover, from Lemma \ref{l1},  $\mathcal{D}_{n_{0}}$ is a finite
set of ${\R}^{d}$. Therefore, $\mathcal{D}_{n_{0}}$ is of the form
$\left\{x_{i}\right\}_{i\in I}$ with $I$ is a finite set. Clearly,
one has $\mu=\sum_{i\in I}a_{i}\delta_{x_{i}}$ with $\sum_{i\in
I}a_{i}=1,$ where
$a_{i}=\mu(\left\{x_{i}\right\}),~~i\in{I}.$\\

Conversely,$~~~$suppose that
$\mu=\sum_{i=1}^{n}\alpha_{i}\delta_{a_{i}}$, where
$a_{i}\in{\R}^{d},\sum_{i=1}^{n}\alpha_{i}=1,\alpha_{i}>0,i=1,2,\dots,n$
and $a_{i}\neq a_{j}$ for all $i\neq j$. Put
$$\Lambda_{n}=\left\{k\in{\N};~~1\leq k\leq\left(
                                             \begin{array}{c}
                                               n+d-1 \\
                                              d-1 \\
                                             \end{array}
                                           \right)
\right\}.$$ Let $$\Big(P_{k,h}\Big)_{\begin{array}{c}
                              0\leq k\leq n \\
                             h\in{\Lambda_{n}}
                            \end{array}
}$$ be an orthogonal basis of ${\Pp}_{n]}$ with respect to the
pre-scalar product on ${\Pp}$ induced by $\mu$: $$\langle
P,Q\rangle_{\mu}=\sum_{i=1}^{n}\alpha_{i}\overline{P}(a_{i})Q(a_{i}),~~P,Q\in{\Pp},$$
with this notation, each $P_{k,h}$ is a polynomial of degree $k$.\\

Now, define the scalar product on ${\R}^{n}$ as follows
$$\Big\langle\left(
               \begin{array}{c}
                 v_{1} \\
                 \vdots \\
                 v_{n} \\
               \end{array}
             \right),\left(
               \begin{array}{c}
                 w_{1} \\
                 \vdots \\
                w_{n} \\
               \end{array}
             \right)
 \Big\rangle=\sum_{i=1}^{n}\alpha_{i}\overline{v}_{i}w_{i}.$$
Therefore, one has
\begin{equation}\label{eqq0}
0=\Big\langle
P_{k_{1},h_{1}},P_{k_{2},h_{2}}\Big\rangle_{\mu}=\Big\langle\left(
               \begin{array}{c}
                 P_{k_{1},h_{1}}(a_{1}) \\
                 \vdots \\
                P_{k_{1},h_{1}}(a_{n}) \\
               \end{array}
             \right),\left(
               \begin{array}{c}
                 P_{k_{2},h_{2}}(a_{1})\\
                 \vdots \\
                P_{k_{2},h_{2}}(a_{n}) \\
               \end{array}
             \right)
 \Big\rangle.
\end{equation}

 for all $0\leq k_{1},k_{2}\leq n$ and all $h_{1}\neq
 h_{2}$ with $h_{1}\in\Lambda_{k_{1}}$ and
 $h_{2}\in\Lambda_{k_{2}}.$\\
\begin{enumerate}
\item[-] First case : if for all $k\in\left\{0,1,\dots,n-1\right\}$
there exists $l_{k}\in\Lambda_{k}$ such that\\
$P_{k,l_{k}}(a_{i_{k}})\neq0$ for some
$i_{k}\in\left\{1,\dots,n\right\}.$ Then, one has
$$\left(\begin{array}{c}
                  P_{k,l_{k}}(a_{1}) \\
                 \vdots\\
                  P_{k,l_{k}}(a_{i_{k}})\\
                 \vdots \\
                  P_{k,l_{k}}(a_{n})\\
               \end{array}
             \right)
            \neq\left(\begin{array}{c}
                 0 \\
               \vdots \\
                 0\\
                 \vdots \\
                 0\\
               \end{array}
             \right).$$

Consider, now the family
$\mathfrak{F}=\left\{P_{0,l_{0}},\dots,P_{n-1,l_{n-1}}\right\}.$ It
is clear that $card (\mathfrak{F})=n$. Put
$$v_{m}=\left(\begin{array}{c}
                  P_{m,l_{m}}(a_{1}) \\
                 \vdots\\
                  P_{m,l_{m}}(a_{i_{m}})\\
                 \vdots \\
                  P_{m,l_{m}}(a_{n})\\
               \end{array}
             \right),~~m\in\left\{0,1,\dots,n-1\right\}.$$
It is clear that $v_{m}\neq0_{{\R}^{n}}$ and $$\langle
v_{m},v_{r}\rangle=0,~~\forall m\neq r.$$Note that
$$\left(\begin{array}{c}
                  P_{n,h}(a_{1}) \\
                 \vdots\\
                  P_{n,h}(a_{n})\\
               \end{array}
             \right)\in{\R}^{n},~~~\forall h\in\Lambda_{n}.$$
Moreover, from (\ref{eqq0}) for all $h\in{\Lambda_{n}}$, one has
$$\Big\langle\left(\begin{array}{c}
                  P_{n,h}(a_{1}) \\
                 \vdots\\
                  P_{n,h}(a_{n})\\
               \end{array}
             \right),v_{j} \Big\rangle=0~~\forall j\in\left\{0,1,\dots,n-1\right\}.$$
It follows that
           $\left(\begin{array}{c}
                  P_{n,h}(a_{1}) \\
                 \vdots\\
                  P_{n,h}(a_{n})\\
               \end{array}
             \right)=\left(\begin{array}{c}
                 0\\
                 \vdots \\
                 0\\
               \end{array}
             \right),~~~\forall h\in{\Lambda_{n}}.$\\
This gives
$$ P_{n,h}(a_{i})=0,~\forall h\in{\Lambda_{n}},~\forall
i\in\left\{1.2,\dots,n\right\}.$$ Therefore, one gets
$$\langle~P_{n,h},P_{n,l}\rangle=0,~~\forall
h,l\in{\Lambda_{n}}.$$ which proves that
$\Omega_{n}=0$  and therefore $\Omega_{m}=0,$ for all $m\geq n$.\\
\item[-] Second case : if there exists
$k_{0}\in\left\{0,1,\dots,n-1\right\}$, such that for all
$l\in{\Lambda_{k_{0}}}$ $$P_{k_{0},l}(a_{i})=0,~~i=1,\dots,n.$$
then, one has
$$\langle~P_{k_{0},h},P_{k_{0},l}\rangle=0,~~\forall
h,l\in{\Lambda_{k_{0}}},$$ which implies that $$\Omega_{k_{0}}=0.$$
\end{enumerate} and therefore $$\Omega_k=0,~~\forall k\geq k_{0}.$$
\end{proof}
%%%%%%%%%%%%%%%%%%%%%%%%%%%%%%%%%%%%%%%%%%%%%%%%%%%%%%%%%%%%%%%%%%%%%%%%%%%%%%%%%%%%%%%%%%%%%%%%%%%%%%%%%%%%%%%%%%%%%%%

\section{Multi-variable orthogonal polynomials }

In the following our purpose is to give the explicit forms of the
Jacobi sequences $(\alpha_{.|n},\Omega_{n})_{n}$ in the case of
Hermite, Laguerre and Jacobi  polynomials.\\

Define the binaire relation ${\Rr}$ on
$\left\{1,2,\dots,d\right\}^{n}$ by
$$(i_{1},i_{2},\dots,i_{n}){\Rr}(j_{1},j_{2},\dots,j_{n})$$ if and only if
$$\left\{i_{1},i_{2},\dots,i_{n}\right\}=\left\{j_{1},j_{2},\dots,j_{n}\right\}$$
and
$$\sharp\Big(\left\{i_{k}=l,~k=1,2,\dots,n\right\}\Big)=\sharp\Big(\left\{j_{k}=l,~k=1,2,\dots,n\right\}\Big)$$
for all $l\in\left\{1,2,\dots,d\right\}.$ ${\Rr}$ is an equivalence
relation on $\left\{1,2,\dots,d\right\}^{n}$ (cf\cite{D.Ab.D.Am.}
for more details). For all $1\leq l\leq d.$ Put
\begin{eqnarray*}
m_{l}&=&\sharp(\left\{i_{k}=l,~k=1,2,\dots,n\right\}).\\
n_{l}&=&\sharp(\left\{j_{k}=l,~k=1,2,\dots,n\right\}).\\
{\Aa}_{n}:&=&\left\{\overline{j}_{n}=cl\Big((j_{1},j_{2},\dots,j_{n})\Big).\right\}\\
e_{\overline{j}_{n}}:&=&e_{j_{1}}\widehat{\otimes}e_{j_{2}}\widehat{\otimes}\dots\widehat{\otimes}e_{j_{n}}.
\end{eqnarray*} where $(e_{i})_{1\leq i\leq d}$ is the canonical basis of
${\C}^{d}.$ It is clear that
$\mathcal{B}=(e_{\overline{j}_{n}})_{\overline{j}_{n}\in{\Aa}_n}$ is
a basis of $({\C}^d)^{\widehat{\otimes} n}$. Moreover, in this basis
the positive definite Jacobi sequence is of form
$\Omega_{n}=(\lambda_{\overline{i}_{n},\overline{j}_{n}})_{\overline{i}_{n},\overline{j}_{n}\in{\Aa}_{n}}.$

\subsection{Basic Notations} Let us introduce the following
notations :
\begin{enumerate}
\item[(1)] If $\beta=(\beta_1,\beta_2,\dots,\beta_d)\in{\N}^d,$
we denote for all $r_1,r_2,\dots,r_d\in{\Z}$
\begin{eqnarray*}
\beta_{r_1,r_2,\dots,r_d}&=&(\beta_1+r_1,\beta_2+r_2,\dots,\beta_d+r_d)\\
\beta_{0,0,\dots,0}&=&\beta\\
\widetilde{0}_{r_1,r_2,\dots,r_d}&=&(0+r_1,0+r_2,\dots,0+r_d)\\
\widetilde{0}&=&0_{{\R}^d}
\end{eqnarray*}
\item[(2)] If $\beta=(\beta_1,\beta_2,\dots,\beta_d)\in{\N}^d$ and
$x=(x_1,x_2,\dots,x_d)\in{\R}^d$, we denote
\begin{eqnarray*}
|\beta|&=&\beta_1+\beta_2+\dots+\beta_d\\
\beta!&=&\beta_1!\beta_2!\dots\beta_d!\\
x^{\beta}&=&x_1^{\beta_1}x_2^{\beta_2}\dots x_d^{\beta_d}\\
|x|_1&=&|x_1|+|x_2|+\dots+|x_d|\\
\|x\|_2&=&\sqrt{x_1^2+x_2^2+\dots+x_d^2}
\end{eqnarray*}
\end{enumerate}

%%%%%%%%%%%%%%%%%%%%%%%%%%%%%%%%%%%%%%%%%%%%%%%%%%%%%%%%%%%%%%%%%%%%%%%%%%%%%%%%%%%%%%%%%%%%%%%%%%%%%%%%%%%%%%%%%%%%%%%%%%%%%%%%%

\subsection{Multiple Hermite polynomials on ${\R}^{d}$}
The multiple Hermite polynomials on ${\R}^{d}$ defined by
$H_{\alpha}=H_{\alpha_{1}}\otimes H_{\alpha_{2}}\otimes\dots\otimes
H_{\alpha_{d}}$ with
$\alpha=(\alpha_{1},\alpha_{2},\dots,\alpha_{d})\in{{\N}^{d}};~|\alpha|=\alpha_{1}+\alpha_{2}+\dots+\alpha_{d}=n$
and for any $i\in\left\{1,2,\dots ,d\right\},~H_{\alpha_{i}}$ is the
classical Hermite polynomial of one variable. For the following
relation we refer the reader to \cite{G.F.D.}.
\begin{equation}\label{eq1}
\frac{d}{dx_{i}}H_{\alpha_{i}}(x_{i})=2\alpha_{i}H_{\alpha_{i}-1}(x_{i}).
\end{equation}
\begin{equation}\label{eq2}
x_{i}H_{\alpha_{i}}(x_{i})=\frac{1}{2}H_{\alpha_{i}+1}(x_{i})+\alpha_{i}H_{\alpha_{i}-1}(x_{i}).
\end{equation}
$$\|H_{\alpha_{i}}\|^{2}=2^{\alpha_{i}}\alpha_{i}!\sqrt{\pi}.$$
It is clear that the multiple Hermite polynomials on ${\R}^{d}$ are
orthogonal with respect to the classical weight function
$$W^{H}(x)=e^{-\|x\|_{_2}^{2}},~~x\in{\R}^{d}.$$

Moreover, the family $(H_{\alpha})_{|\alpha|=n}$ is an orthogonal
basis of ${\Pp}_{n}$ with respect to $\mu$, where $\mu$ is the
measure of density $W^{H}$ with respect to the Lebesgue measure on
${\R}^d$. For
$\alpha=(\alpha_{1},\alpha_{2},\dots,\alpha_{d})\in{{\N}^{d}};~|\alpha|=n$,
one has
\begin{equation}\label{eq3}
H_{\alpha}(x)=H_{\alpha_{1}}(x_{1})H_{\alpha_{2}}(x_{2})\dots
H_{\alpha_{d}}(x_{d}),~\forall
x=(x_{1},x_{2},\dots,x_{d})\in{\R}^{d}
\end{equation} Multiplying both sides in (\ref{eq3}) by $x_{i}$ and
using (\ref{eq2}) one gets
$$x_{i}H_{\alpha}(x)=\frac{1}{2}H_{(\alpha_{1},..,\alpha_{i-1},\alpha_{i}+1,\alpha_{i+1},\dots,\alpha_{d})}(x)+\alpha_{i}
H_{(\alpha_{1},\dots,\alpha_{i-1},\alpha_{i}-1,\alpha_{i+1},\dots,\alpha_{d})}(x).$$
From the above notations, it follows that
\begin{equation}\label{eq4}
X_{i}H_{\alpha}=\frac{1}{2}H_{\alpha_{_{0,\dots,0,1,0,\dots,0}}}+\alpha_{i}
H_{\alpha_{_{0,\dots,0,-1,0,\dots,0}}}
\end{equation}
where $1$ and $-1$ are in the $i-th$ index. Note that
\begin{equation}\label{eqnh}
\|H_{\alpha}\|^{2}=\prod_{i=1}^{d}\|H_{\alpha_{i}}\|^{2}=2^{n}\pi^{\frac{d}{2}}\alpha!
\end{equation}

Now, consider the orthogonal projector from $\mathcal{P}$ to
$\mathcal{P}_{n}$ given by
\begin{eqnarray*}
P_{n}:&=&\sum_{|\alpha|=n}\frac{1}{\|H_\alpha\|^2}|H_{\alpha}\rangle\langle H_{\alpha}|, \ \ n\in{\N} \\
P_{-1}:&=&0.
\end{eqnarray*}
For $i\in\left\{1,2,\dots,d\right\}$, define the CAP operators as
follows :
\begin{eqnarray*}
a^{+}_{i|n}:&=&P_{n+1}X_{i}P_{n}\\
&=&\sum_{|\beta|=n+1,|\alpha|=n}\frac{1}{\|H_\alpha\|^2\|H_\beta\|^2}|H_{\beta}\rangle\langle H_{\beta}|X_{i}|H_{\alpha}\rangle\langle H_{\alpha}|\\
&=&\sum_{|\beta|=n+1,|\alpha|=n}\frac{1}{\|H_\alpha\|^2\|H_\beta\|^2}\langle
H_{\beta},X_{i}H_{\alpha}\rangle_{\mu} |H_{\beta}\rangle\langle H_{\alpha}|.\\
a^{0}_{i|n}:&=&P_{n}X_{i}P_{n}\\&=&\sum_{|\beta|=n,|\alpha|=n}\frac{1}{\|H_\alpha\|^2\|H_\beta\|^2}\langle
H_{\beta}, X_{i}H_{\alpha}\rangle_{\mu} |H_{\beta}\rangle\langle
H_{\alpha}|.\\
a^{-}_{i|n}:&=&P_{n-1}X_{i}P_{n}, \ \ n\geq1  \\
&=&\sum_{|\beta|=n-1,|\alpha|=n}\frac{1}{\|H_\alpha\|^2\|H_\beta\|^2}\langle
H_{\beta}, X_{i}H_{\alpha}\rangle_{\mu} |H_{\beta}\rangle\langle
H_{\alpha}|.
\end{eqnarray*}
From (\ref{eq4}), it follows that
\begin{eqnarray*}
a^{+}_{i|n}&=&\frac{1}{2}\sum_{|\alpha|=n}\frac{1}{\|H_\alpha\|^2}|H_{\alpha_{_{0,\dots,0,1,0,\dots,0}}}\rangle\langle H_{\alpha}|\\
a^{0}_{i|n}&=&0\\
a^{-}_{i|n}&=&\sum_{|\alpha|=n}\frac{\alpha_i}{\|H_\alpha\|^2}|H_{\alpha_{_{0,\dots,0,-1,0,\dots,0}}}\rangle\langle
H_{\alpha}|,~~n\geq1~~~~(a^-_{i|0}:=0).
\end{eqnarray*}
Then, for all $\alpha,\beta\in{\N}^{d}$ such that $|\alpha|=n$ and
$|\beta|=n+1,$ one has
\begin{equation}\label{eq5}
a^{+}_{i|n}H_{\alpha}=\frac{1}{2}H_{\alpha_{_{0,\dots,0,1,0,\dots,0}}}.
\end{equation}
\begin{equation}\label{eq6}
a^{-}_{i|n+1}H_{\beta}=\beta_iH_{\beta_{_{0,\dots,0,-1,0,\dots,0}}}.
\end{equation}
where $1$ and $-1$ are in the $i-th$ index. Now, for all
$k\in\left\{1,2,\dots,d\right\},$ put
$$a^{+}_{k}=\sum_{n\in{\N}}a^{+}_{k|n}.$$
\begin{lem}\label{lh1} For all $1\leq k\leq d,m\in{\N}^{*}$ and
$\alpha=(\alpha_{1},\alpha_{2},\dots,\alpha_{d})$ such that
$|\alpha|=n,$ one has :
\begin{equation}\label{eql2}
(a^+_k)^mH_\alpha=\Big(\frac{1}{2}\Big)^mH_{\alpha_{0,\dots,0,m,0,\dots,0}}
\end{equation}
where
\begin{eqnarray*}
\alpha_{0,\dots,0,m,0,\dots,0}&=&(\alpha_{1},\dots,\alpha_{k-1},(\alpha_{k}+m),\alpha_{k+1},\dots,\alpha_{d}).
\end{eqnarray*}
\end{lem}
\begin{proof} We prove the above lemma by induction on $m\in{\N}^*.$
\begin{enumerate}
\item[-] For $m=1,$ one has $$a^{+}_{k}H_{\alpha}=\frac{1}{2}H_{\alpha_{0,\dots,0,1,0,\dots,0}}$$
\item[-] Let $m\geq1$ and suppose that (\ref{eql2}) holds true.
Then, one has
\begin{eqnarray*}
(a^{+}_{k})^{m+1}H_{\alpha}&=& a^{+}_{k}(a^{+}_{k})^{m}H_{\alpha}\\
&=&\Big(\frac{1}{2}\Big)^ma_k^+H_{\alpha_{0,\dots,0,m,0,\dots,0}}\\
&=&\Big(\frac{1}{2}\Big)^m\frac{1}{2}H_{\alpha_{0,\dots,0,m+1,0,\dots,0}}\\
&=&\Big(\frac{1}{2}\Big)^{m+1}H_{\alpha_{0,\dots,0,m+1,0,\dots,0}}
\end{eqnarray*}
\end{enumerate}
This ends the proof.
\end{proof}
\begin{theo} For all $n\in{\N}$, one has $$\alpha_{.|n}\equiv0$$ and
the coefficients of $\Omega_n$ in the basis $\mathcal{B}$ are given
by
$$\lambda_{\overline{i}_{n},\overline{j}_{n}}=\delta_{\overline{i}_{n},\overline{j}_{n}}\Big(\frac{1}{2}\Big)^{|\overline{n}|}\pi^{\frac{d}{2}}\overline{n}!$$
where
\begin{eqnarray*}
\overline{i}_n&=& cl\Big(({i}_{1},\dots,{i}_{n})\Big)\\
\overline{j}_n&=& cl\Big(({j}_{1},\dots,{j}_{n})\Big)
\end{eqnarray*}
and
$$n_{l}=\sharp\Big(\left\{{i}_{k}=l,k=1,\dots,n\right\}\Big),(1\leq l\leq d),\overline{n}=(n_{1},n_{2},\dots,n_{d}).$$
\end{theo}
\begin{proof}  Because $a^{0}_{.|n}=0$, for all $n\in{\N}$, then one
has
$$\alpha_{.|n}=U_{n}^{-1}a^{0}_{.|n}U_{n}=0.$$

Now, recall that
\begin{eqnarray*}
\lambda_{\overline{i}_{n},\overline{j}_{n}}&=&\langle e_{\overline{i}_{n}},\Omega_{n} e_{\overline{j}_{n}}\rangle_{({\C}^{d})^{\widehat{\otimes}n}}\\
&=& \langle
e_{i_{1}}\widehat{\otimes}e_{i_{2}}\widehat{\otimes}\dots\widehat{\otimes}e_{i_{n}},
\Omega_{n}e_{j_{1}}\widehat{\otimes}e_{j_{2}}\widehat{\otimes}\dots\widehat{\otimes}e_{j_{n}}\rangle_{({\C}^{d})^{\widehat{\otimes}n}}\\
&=&\langle a^{+}_{i_{1}}a^{+}_{i_{2}}\dots a^{+}_{i_{n}}\Phi,a^{+}_{j_{1}}a^{+}_{j_{2}}\dots a^{+}_{j_{n}}\Phi\rangle_\mu\\
&=&\langle(a^{+}_{1})^{m_{1}}(a^{+}_{2})^{m_{2}}\dots(a^{+}_{d})^{m_{d}}\Phi,(a^{+}_{1})^{n_{1}}(a^{+}_{2})^{n_{2}}\dots(a^{+}_{d})^{n_{d}}\Phi\rangle_{\mu}
\end{eqnarray*}
where $\Phi=1_{\Pp}=H_{\widetilde{0}}$, $\widetilde{0}=0_{{\R}^{d}}$
and
$m_{l}=\sharp\left\{i_k=l,~k=1,\dots,n\right\},n_{l}=\sharp\left\{j_k=l,~k=1,\dots,n\right\}$
for all $1\leq l\leq d$. Then, from Lemma \ref{lh1}
\begin{eqnarray*}
(a^{+}_{d-1})^{n_{d-1}}(a^{+}_{d})^{n_{d}}\Phi&=&\Big(\frac{1}{2}\Big)^{n_d}(a^{+}_{d-1})^{n_{d-1}}H_{\widetilde{0}_{0,\dots,0,n_d}}\\
&=&\Big(\frac{1}{2}\Big)^{n_{d-1}+n_d}H_{\widetilde{0}_{0,\dots,0,n_{d-1},n_d}}.
\end{eqnarray*}
 Repeating the above argument until to obtain
\begin{eqnarray*}
(a^{+}_{1})^{n_{1}}\dots(a^{+}_{d})^{n_{d}}\Phi&=&\Big(\frac{1}{2}\Big)^{|\overline{n}|}H_{\overline{n}}.
\end{eqnarray*}
where $\overline{n}=(n_1,n_2,\dots,n_d)$.
\begin{enumerate}
\item[(i)] If $\overline{i}_{n}=\overline{j}_{n},$ then one has
\begin{eqnarray*}
\lambda_{\overline{i}_{n},\overline{j}_{n}}
&=&\langle(a^{+}_{1})^{m_{1}}(a^{+}_{2})^{m_{2}}\dots(a^{+}_{d})^{m_{d}}\Phi,(a^{+}_{1})^{n_{1}}(a^{+}_{2})^{n_{2}}\dots(a^{+}_{d})^{n_{d}}\Phi\rangle_{\mu}\\
&=&\Big(\frac{1}{2}\Big)^{2|\overline{n}|}\|H_{\overline{n}}\|^2\\
&=&\Big(\frac{1}{2}\Big)^{2|\overline{n}|}2^{|\overline{n}|}\pi^{\frac{d}{2}}\overline{n}!\\
&=&\Big(\frac{1}{2}\Big)^{|\overline{n}|}\pi^{\frac{d}{2}}\overline{n}!
\end{eqnarray*}
\item[(ii)] If $\overline{i}_{n}\neq\overline{j}_{n},$ then,
$\left\{{i}_{1},\dots,{i}_{n}\right\}\neq\left\{{j}_{1},\dots,{j}_{n}\right\}$
or there exists $l\in\left\{{i}_{1},\dots,{i}_{n}\right\}$ such that
$m_l\neq n_l$.

\item[-] First case : if
$\left\{{i}_{1},\dots,{i}_{n}\right\}\neq\left\{{j}_{1},\dots,{j}_{n}\right\}$,
then there exists $l\in\left\{1,\dots,d\right\}$ such that
$l\in\left\{{i}_{1},\dots,{i}_{n}\right\}$ and
$l\not\in\left\{{j}_{1},\dots,{j}_{n}\right\}$ or the converse.
Without loss of generality suppose that $l=1$ i.e. $m_{1}\neq0$ and
$n_{1}=0.$ Therefore, one gets
\begin{eqnarray*}
\lambda_{\overline{i}_{n},\overline{j}_{n}}&=&\Big(\frac{1}{2}\Big)^{(|\overline{m}|+|\overline{n}|)}\langle H_{\overline{m}}, H_{\overline{n}}\rangle_{\mu}\\
&=&0.
\end{eqnarray*} because $ H_{\overline{m}}$ and $H_{\overline{n}}$
are orthogonal $(\overline{m}\neq\overline{ n}).$

\item[-] Second case : if there exists
$l\in\left\{1,\dots,d\right\}$ such that $m_{l}\neq
n_{l}~i.e.~\overline{m}\neq\overline{ n},$ then, one has
$$\langle H_{\overline{m}}, H_{\overline{n}}\rangle_{\mu}=0.$$  It
follows that
$$\lambda_{\overline{i}_{n},\overline{j}_{n}}=0.$$
\end{enumerate}
\end{proof}

%%%%%%%%%%%%%%%%%%%%%%%%%%%%%%%%%%%%%%%%%%%%%%%%%%%%%%%%%%%%%%%%%%%%%%%%%%%%%%%%%%%%%%%%%%%%%%%%%%%%%%%%%%%%%%%%%%%%

\section{Multiple Laguerre polynomials on ${\R}^d_+$}

As in the multiple Hermite polynomials on ${\R}^d$. The multiple
Laguerre polynomials on ${\R}^d_+$ with parameter
$\alpha=(\alpha_1,\alpha_2,\dots,\alpha_d);~\alpha_j>-1,~j=1,2,\dots,d$
are defined as follows
$$L^\alpha_k=L^{\alpha_1}_{k_1}\otimes L^{\alpha_2}_{k_2}\otimes\dots\otimes L^{\alpha_d}_{k_d}$$
where, $k=(k_1,k_2,\dots,k_d)\in{\N}^d$ such that $|k|=n$ and for
any $i\in\left\{1,2,\dots,d\right\},~L^{\alpha_i}_{k_i}$ is the
classical Laguerre polynomials on ${\R}_+$. For the following
relation we refer the reader to \cite{G.F.D.}.

\begin{equation}\label{eqqql0}
x_iL^{\alpha_i}_{k_i}(x_i)=-(k_i+1)L^{\alpha_i}_{k_i+1}(x_i)+(2k_i+\alpha_i+1)L^{\alpha_i}_{k_i}(x_i)-(k_i+\alpha_i)L^{\alpha_i}_{k_i-1}(x_i)
\end{equation}
$$\|L^{\alpha_i}_{k_i}\|^2=\frac{\Gamma(\alpha_i+k_i+1)}{k_i!}.$$
where $\Gamma$ is the Gamma function defined by
$$\Gamma(y)=\int_0^\infty t^{y-1}e^{-t}dt,\ \ \ \ \forall y>0. $$

It is clear that the multiple Laguerre polynomials are orthogonal
with respect the weight function $$W^L_{\alpha}(x)=x^\alpha
e^{-|x|_1},~x=(x_1,\dots,x_d)\in{\R}^d_+$$

Moreover, the family $(L^\alpha_k)_{|k|=n}$ is an orthogonal basis
of ${\Pp}_n$ with respect to $\mu$, where $\mu$ is the measure of
density $W^L_{\alpha}$ with respect to the Lebesgue measure on
${\R}^d_+$. For $k=(k_1,k_2,\dots,k_d);|k|=n$, one has
\begin{equation}\label{eql1}
L^\alpha_k(x)=L^{\alpha_1}_{k_1}(x_1) L^{\alpha_2}_{k_2}(x_2)\dots
L^{\alpha_d}_{k_d}(x_d).
\end{equation}
Multiplying both sides in (\ref{eql1}) by $x_i$ and using
(\ref{eqqql0}), one gets
\begin{eqnarray*}
x_iL^\alpha_k(x)&=&-(k_i+1)L^{\alpha}_{(k_1,\dots,k_{i-1},k_i+1,k_{i+1},\dots,k_d)}(x)+(2k_i+\alpha_i+1)L^\alpha_k(x)\\&&-(k_i+\alpha_i)
L^{\alpha}_{(k_1,\dots,k_{i-1}k_i-1,k_{i+1},\dots,k_d)}(x).
\end{eqnarray*}
From the above notations, it follows that
\begin{equation}\label{eql2}
X_iL^\alpha_k=-(k_i+1)L^{\alpha}_{k_{0,\dots,0,1,0,\dots,0}}+(2k_i+\alpha_i+1)L^\alpha_k-(k_i+\alpha_i)
L^{\alpha}_{k_{0,\dots,0,-1,0,\dots,0}}
\end{equation} where $-1,1$ are in the $i-th$ index. Note that

$$\|L^\alpha_k\|^2=\prod_{j=1}^d\|L^{\alpha_j}_{k_j}\|=\frac{1}{k!}\prod_{j=1}^d\Gamma(k_j+\alpha_j+1)$$

Now, consider the orthogonal projector from $\mathcal{P}$ to
$\mathcal{P}_{n}$ given by
\begin{eqnarray*}
P_{n}:&=&\sum_{|\alpha|=n}\frac{1}{\|L^\alpha_k\|^2}|L^\alpha_k\rangle\langle L^\alpha_k|,\ \ n\in{\N}\\
P_{-1}:&=&0
\end{eqnarray*}
For $i\in\left\{1,2,\dots,d\right\}$, define the CAP operators as
follows :
\begin{eqnarray*}
a^{+}_{i|n}:&=&P_{n+1}X_{i}P_{n}\\
&=&\sum_{|\beta|=n+1,|k|=n}\frac{1}{\|L^\alpha_\beta\|^2\|L^\alpha_k\|^2}|L^\alpha_\beta\rangle\langle L^\alpha_\beta|X_i|L^\alpha_k\rangle\langle L^\alpha_k|\\
&=&\sum_{|\beta|=n+1,|k|=n}\frac{1}{\|L^\alpha_\beta\|^2\|L^\alpha_k\|^2}\langle
L^\alpha_\beta,X_{i}L^\alpha_k\rangle_{\mu} |L^\alpha_\beta\rangle\langle L^\alpha_k|\\
a^{0}_{i|n}:&=&P_{n}X_iP_{n}\\&=&\sum_{|\beta|=n,|k|=n}\frac{1}{\|L^\alpha_\beta\|^2\|L^\alpha_k\|^2}\langle
L^\alpha_\beta,X_iL^\alpha_k\rangle_{\mu}|L^\alpha_\beta\rangle\langle L^\alpha_k|\\
a^{-}_{i|n}:&=&P_{n-1}X_{i}P_{n},~~~~n\geq1\\
&=&\sum_{|\beta|=n-1,|k|=n}\frac{1}{\|L^\alpha_\beta\|^2\|L^\alpha_k\|^2}\langle
L^\alpha_\beta,X_iL^\alpha_k\rangle_{\mu}|L^\alpha_\beta\rangle\langle
L^\alpha_k|.
\end{eqnarray*}
From (\ref{eql2}), it follows that
\begin{eqnarray*}
a^{+}_{i|n}&=&-\sum_{|k|=n}\frac{(k_i+1)}{\|L^\alpha_k\|^2}|L^{\alpha}_{k_{0,\dots,0,1,0,\dots,0}}\rangle\langle L^\alpha_k|\\
a^{-}_{i|n}&=&-\sum_{|k|=n}\frac{(k_i+\alpha_i)}{\|L^\alpha_k\|^2}|L^{\alpha}_{k_{0,\dots,0,-1,0,\dots,0}}\rangle\langle
L^\alpha_k|,
\ \ n\geq1 \ \ (a^-_{i|0}:=0) \\
a^{0}_{i|n}&=&\sum_{|k|=n}\frac{(2k_i+\alpha_i+1)}{\|L^\alpha_k\|^2}|L^\alpha_k\rangle\langle
L^\alpha_k|.
\end{eqnarray*}
Then, for all $k,\beta\in{\N}^d$ such that $|k|=n$ and
$|\beta|=n+1$, one has
\begin{eqnarray}\label{eql8}
a^{+}_{i|n}L^\alpha_k&=&-(k_i+1)L^{\alpha}_{k_{0,\dots,0,1,0,\dots,0}}\nonumber \\
a^{-}_{i|n+1}L^\alpha_\beta&=&-(k_i+\alpha_i)L^{\alpha}_{\beta_{0,\dots,0,-1,0,\dots,0}}\\
a^{0}_{i|n}L^\alpha_k&=&(2k_i+\alpha_i+1)L^\alpha_k \nonumber
\end{eqnarray}
where $1,-1$ are in the $i-th$ index.
\begin{lem}\label{Ll1}
For all $1\leq i\leq d,m\in{\N}^*$ and
$k=(k_1,k_2,\dots,k_d)\in{\N}^d$ such that $|k|=n$, one has
\begin{equation}\label{eql9}
(a^+_i)^mL^\alpha_k=(-1)^m\prod_{p=1}^m(k_i+p)L^{\alpha}_{k_{0,\dots,0,m,0,\dots,0}}
\end{equation}
where
$$k_{0,\dots,0,m,0,\dots,0}=(k_1,\dots,k_{i-1},k_{i}+m,k_{i+1},\dots,k_d).$$
\end{lem}
\begin{proof} We prove the above lemma by induction on $m\in{\N}^*$.
\begin{enumerate}
\item[-] For $m=1$, one has $$a^+_iL^\alpha_k=-(k_i+1)L^{\alpha}_{k_{0,\dots,0,1,0,\dots,0}}$$
\item[-] Let $m\geq 1$ and suppose that (\ref{eql9}) holds true.
Then, one has
\begin{eqnarray*}
(a^+_i)^{m+1}L^\alpha_k&=&a^+_i(a^+_i)^mL^\alpha_k\\
&=&(-1)^m\prod_{p=1}^m(k_i+p)a^+_iL^{\alpha}_{k_{0,\dots,0,m,0,\dots,0}}\\
&=&(-1)^{m+1}\prod_{p=1}^m(k_i+p)(k_i+m+1)L^{\alpha}_{k_{0,\dots,0,m+1,0,\dots,0}}\\
&=&(-1)^{m+1}\prod_{p=1}^{m+1}(k_i+p)L^{\alpha}_{k_{0,\dots,0,m+1,0,\dots,0}}
\end{eqnarray*}
\end{enumerate}
\end{proof}
\begin{theo}  For all $n\in{\N}$ and $\overline{i}_n=cl\Big(({i}_{1},\dots,{i}_{n})\Big),\overline{j}_n=cl\Big(({j}_{1},\dots,{j}_{n})\Big)\in{\Aa}_n$,
we have
$$\alpha_{e_l|n}e_{\overline{i}_n}=(2n_l+\alpha_l+1)e_{\overline{i}_n}$$ and
the coefficients of $\Omega_n$ in the basis
$\mathcal{B}=(e_{\overline{i}_n})_{\overline{i}_n\in{\Aa}_n}$ are
given by
$$\lambda_{\overline{i}_{n},\overline{j}_{n}}=\delta_{\overline{i}_n,\overline{j}_n}\overline{n}!
\prod_{l=1}^d\Gamma(n_l+\alpha_l+1)$$ where
$$n_{l}=\sharp\Big(\left\{{i}_{k}=l,k=1,\dots,n\right\}\Big),(1\leq
l\leq d),\overline{n}=(n_{1},n_{2},\dots,n_{d}).$$
\end{theo}
\begin{proof}
Recall that
\begin{eqnarray*}
\lambda_{\overline{i}_{n},\overline{j}_{n}}&=&\langle e_{\overline{i}_{n}},\Omega_{n} e_{\overline{j}_{n}}\rangle_{({\C}^{d})^{\widehat{\otimes}n}}\\
&=& \langle
e_{i_{1}}\widehat{\otimes}e_{i_{2}}\widehat{\otimes}\dots\widehat{\otimes}e_{i_{n}},
\Omega_{n}e_{j_{1}}\widehat{\otimes}e_{j_{2}}\widehat{\otimes}\dots\widehat{\otimes}e_{j_{n}}\rangle_{({\C}^{d})^{\widehat{\otimes}n}}\\
&=&\langle a^{+}_{i_{1}}a^{+}_{i_{2}}\dots a^{+}_{i_{n}}\Phi,a^{+}_{j_{1}}a^{+}_{j_{2}}\dots a^{+}_{j_{n}}\Phi\rangle_\mu\\
&=&\langle(a^{+}_{1})^{m_{1}}(a^{+}_{2})^{m_{2}}\dots(a^{+}_{d})^{m_{d}}\Phi,(a^{+}_{1})^{n_{1}}(a^{+}_{2})^{n_{2}}\dots(a^{+}_{d})^{n_{d}}\Phi\rangle_{\mu}
\end{eqnarray*}
where $\Phi=1_{\Pp}=H_{\widetilde{0}}$, $\widetilde{0}=0_{{\R}^{d}}$
and
$m_{l}=\sharp\left\{i_k=l,~k=1,\dots,n\right\},n_{l}=\sharp\left\{j_k=l,~k=1,\dots,n\right\}$
for all $1\leq l\leq d$. On the other hand, from Lemma \ref{Ll1},
one has
\begin{eqnarray*}
(a^+_{d-1})^{n_{d-1}}(a^+_d)^{n_d}\Phi&=&(-1)^{n_d}n_d!(a^+_{d-1})^{n_{d-1}}L^\alpha_{\widetilde{0}_{0,\dots,0,n_d}}\\
&=&(-1)^{n_{d-1}+n_d}n_{d-1}!n_d!L^\alpha_{\widetilde{0}_{0,\dots,0,n_{d-1},n_d}}.
\end{eqnarray*}
Repeating the above argument until to obtain
\begin{equation}\label{eqqL1}
(a_1^+)^{n_1}\dots(a_d^+)^{n_d}\Phi=(-1)^{|\overline{n}|}\overline{n}!L^\alpha_{\overline{n}}
\end{equation}
where $\overline{n}=(n_1,n_2,\dots,n_d)$.
\begin{enumerate}
\item[(i)] If $\overline{i}_n=\overline{j}_n$, then, one has
\begin{eqnarray*}
\lambda_{\overline{i}_{n},\overline{j}_{n}}
&=&\langle(a^{+}_{1})^{m_{1}}(a^{+}_{2})^{m_{2}}\dots(a^{+}_{d})^{m_{d}}\Phi,(a^{+}_{1})^{n_{1}}(a^{+}_{2})^{n_{2}}\dots(a^{+}_{d})^{n_{d}}\Phi\rangle_{\mu}.\\
&=&(\overline{n}!)^2\|L^\alpha_{\overline{n}}\|^2\\
&=&\overline{n}!\prod_{l=1}^d\Gamma(n_l+\alpha_l+1)
\end{eqnarray*}
\item[(ii)] If $\overline{i}_n\neq\overline{j}_n$, then
$\left\{{i}_{1},\dots,{i}_{n}\right\}\neq\left\{{j}_{1},\dots,{j}_{n}\right\}$
or there exists $l\in\left\{{i}_{1},\dots,{i}_{n}\right\}$ such that
$m_l\neq n_l$.

\item[-] First case : if
$\left\{{i}_{1},\dots,{i}_{n}\right\}\neq\left\{{j}_{1},\dots,{j}_{n}\right\}$,
then there exists $l\in\left\{1,\dots,d\right\}$ such that
$l\in\left\{{i}_{1},\dots,{i}_{n}\right\}$ and
$l\not\in\left\{{j}_{1},\dots,{j}_{n}\right\}$ or the converse.
Without loss of generality suppose that $l=1$ i.e. $m_{1}\neq0$ and
$n_{1}=0.$ Therefore, one has
\begin{eqnarray*}
\lambda_{\overline{i}_{n},\overline{j}_{n}}&=&(-1)^{|\overline{m}|+|\overline{n}|}\overline{m}!\overline{n}!
\langle L^\alpha_{\overline{m}},L^\alpha_{\overline{n}}\rangle_\mu\\
&=&0
\end{eqnarray*}
because $ L^\alpha_{\overline{m}}$ and $L^\alpha_{\overline{n}}$ are
orthogonal $(\overline{m}\neq\overline{ n}).$

\item[-] Second case : if there exists
$l\in\left\{1,\dots,d\right\}$ such that $m_{l}\neq
n_{l}~i.e.~\overline{m}\neq\overline{ n},$ then, one gets
$$\langle L^\alpha_{\overline{m}}, L^\alpha_{\overline{n}}\rangle_{\mu}=0.$$  It follows that
$$\lambda_{\overline{i}_{n},\overline{j}_{n}}=0.$$

Now, let $\overline{i}_n=cl\Big((i_1,i_2,\dots,i_d)\Big)\in{\Aa}_n$.
Recall that
$$U_ne_{\overline{i}_n}:=a^+_{i_1}a^+_{i_2}\dots a^+_{i_d}\Phi.$$
Then, from identities (\ref{eqqL1}) and (\ref{eql8}), it follows
that for all $l\in\left\{1,2,\dots,d\right\}$
\begin{eqnarray*}
\alpha_{e_l|n}e_{\overline{i}_n}&:=&U_n^{-1}a^0_{l|n}U_ne_{\overline{i}_n}\\
&=&U^{-1}_na^0_{l|n}a^+_{i_1}a^+_{i_2}\dots a^+_{i_d}\Phi\\
&=&U^{-1}_na^0_{l|n}(a^+_1)^{n_1}\dots(a^+_d)^{n_d}\Phi\\
&=&(-1)^{|\overline{n}|}\overline{n}!U^{-1}_na^0_{l|n}L^\alpha_{\overline{n}}\\
&=&(-1)^{|\overline{n}|}\overline{n}!(2n_l+\alpha_l+1)U^{-1}_nL^\alpha_{\overline{n}}\\
&=&(2n_l+\alpha_l+1)U^{-1}_n(a^+_1)^{n_1}\dots(a^+_d)^{n_d}\Phi\\
&=&(2n_l+\alpha_l+1)U^{-1}_na^+_{i_1}a^+_{i_2}\dots a^+_{i_d}\Phi\\
&=&(2n_l+\alpha_l+1)e_{\overline{i}_n}
\end{eqnarray*}
where $n_j=\sharp\left\{i_p=j; \ \ p=1,2,\dots,d\right\},\ \ 1\leq
j\leq d$.
\end{enumerate}
\end{proof}

%%%%%%%%%%%%%%%%%%%%%%%%%%%%%%%%%%%%%%%%%%%%%%%%%%%%%%%%%%%%%%%%%%%%%%%%%%%%%%%%%%%%%%%%%%%%%%%%%%%%%%%%%%%%%%%%%%%%

\section{Multiple Jacobi  polynomials on the cube}
The multiple Jacobi  polynomials on the cube $[-1,1]^d$ with
parameter
$a=(a_1,a_2,\dots,a_d),\\b=(b_1,b_2,\dots,b_d);~a_j>-1,b_j>-1,~j=1,2,\dots,d$
are defined as follows
$$P^{(a,b)}_\alpha=P^{(a_1,b_1)}_{\alpha_1}\otimes P^{(a_2,b_2)}_{\alpha_2}\otimes\dots\otimes P^{(a_d,b_d)}_{\alpha_d}.$$
where, $\alpha=(\alpha_1,\alpha_2,\dots,\alpha_d)\in{\N}^d$ such
that $|\alpha|=n$ and for any
$i\in\left\{1,2,\dots,d\right\},P^{(a_i,b_i)}_{\alpha_i}$ is the
classical Jacobi polynomials on $[-1,1]$. For the following
relations we refer the reader to \cite{G.F.D.}.
\begin{eqnarray*}
x_iP^{(a_i,b_i)}_{\alpha_i}(x_i)&=&\frac{2(\alpha_i+1)(\alpha_i+b_i+a_i+1)}{(2\alpha_i+b_i+a_i+1)(2\alpha_i+b_i+a_i+2)}P^{(a_i,b_i)}_{\alpha_i+1}(x_i)\\
&-&\frac{(a_i^2-b_i^2)}{(2\alpha_i+b_i+a_i)(2\alpha_i+b_i+a_i+2)}P^{(a_i,b_i)}_{\alpha_i}(x_i)\\
&+&\frac{2(\alpha_i+a_i)(\alpha_i+b_i)}{(2\alpha_i+b_i+a_i)(2\alpha_i+b_i+a_i+1)}P^{(a_i,b_i)}_{\alpha_i-1}(x_i)
\end{eqnarray*}
$$\|P^{(a_i,b_i)}_{\alpha_i}\|^2=\frac{2^{b_i+a_i+1}\Gamma(\alpha_i+a_i+1)\Gamma(\alpha_i+b_i+1)}{\alpha_i!(2\alpha_i+b_i+a_i+1)\Gamma(\alpha_i+b_i+a_i+1)}$$
where $\Gamma$ is the Gamma function defined by
$$\Gamma(y)=\int_0^\infty t^{y-1}e^{-t}dt,\ \ \ \ \forall y>0. $$

It is clear that the multiple Jacobi polynomials on the cube
$[-1,1]^d$ are orthogonal with respect the weight function
$$W^J_{a,b}(x)=\prod_{j=1}^{d}(1-x_j)^{a_j}(1+x_j)^{b_j}, \ \ x=(x_1,\dots,x_d)\in[-1,1]^d.$$

Moreover, the family $(P^{(a,b)}_\alpha)_{|\alpha|=n}$ is an
orthogonal basis of ${\Pp}_n$ with respect to $\mu$, where $\mu$ is
the measure of density $W^J_{a,b}$ with respect to the Lebesgue
measure on $[-1,1]^d$. For
$\alpha=(\alpha_1,\alpha_2,\dots,\alpha_d);|\alpha|=n$, one has
\begin{equation}\label{eqJ2}
P^{(a,b)}_\alpha(x)=P^{(a_1,b_1)}_{\alpha_1}(x_1)
P^{(a_2,b_2)}_{\alpha_2}(x_2)\dots P^{(a_d,b_d)}_{\alpha_d}(x_d).
\end{equation}
Multiplying both sides in (\ref{eqJ2}) by $x_i$, one gets
\begin{eqnarray*}
x_iP^{(a,b)}_{\alpha}(x)&=&\frac{2(\alpha_i+1)(\alpha_i+b_i+a_i+1)}{(2\alpha_i+b_i+a_i+1)(2\alpha_i+b_i+a_i+2)}
P^{(a,b)}_{(\alpha_{1},\dots,\alpha_{i-1},\alpha_{i}+1,\alpha_{i+1},\dots,\alpha_{d})}(x)\\
&-&\frac{(a_i^2-b_i^2)}{(2\alpha_i+b_i+a_i)(2\alpha_i+b_i+a_i+2)}P^{(a,b)}_{\alpha}(x)\\
&+&\frac{2(\alpha_i+a_i)(\alpha_i+b_i)}{(2\alpha_i+b_i+a_i)(2\alpha_i+b_i+a_i+1)}
P^{(a,b)}_{(\alpha_{1},\dots,\alpha_{i-1},\alpha_{i}-1,\alpha_{i+1},\dots,\alpha_{d})}(x)
\end{eqnarray*}
From the above Notation, it follows that
\begin{eqnarray}\label{*}
X_iP^{(a,b)}_{\alpha}&=&\frac{2(\alpha_i+1)(\alpha_i+b_i+a_i+1)}{(2\alpha_i+b_i+a_i+1)(2\alpha_i+b_i+a_i+2)}
P^{(a,b)}_{\alpha_{0,\dots,1,0,\dots,0}}\nonumber\\
&-&\frac{(a_i^2-b_i^2)}{(2\alpha_i+b_i+a_i)(2\alpha_i+b_i+a_i+2)}P^{(a,b)}_{\alpha}\\
&+&\frac{2(\alpha_i+a_i)(\alpha_i+b_i)}{(2\alpha_i+b_i+a_i)(2\alpha_i+b_i+a_i+1)}
P^{(a,b)}_{\alpha_{0,\dots,-1,0,\dots,0}}\nonumber
\end{eqnarray}
where $-1,1$ are in the $i-th$ index. Note that
\begin{eqnarray*}
\|P^{(a,b)}_{\alpha}(x)\|^2&=&\prod_{i=1}^d\|P^{(a_i,b_i)}_{\alpha_i}\|^2\\
&=&\frac{1}{\alpha!}\prod_{i=1}^{d}\frac{2^{b_i+a_i+1}\Gamma(\alpha_i+a_i+1)\Gamma(\alpha_i+b_i+1)}{(2\alpha_i+b_i+a_i+1)\Gamma(\alpha_i+b_i+a_i+1)}
\end{eqnarray*}

Now, consider the orthogonal projector from $\mathcal{P}$ to
$\mathcal{P}_{n}$ given by
\begin{eqnarray*}
P_{n}:&=&\sum_{|\alpha|=n}\frac{1}{\|L^\alpha_k\|^2}|L^\alpha_k\rangle\langle L^\alpha_k|,\ \ n\in{\N}\\
P_{-1}:&=&0
\end{eqnarray*}
For $i\in\left\{1,2,\dots,d\right\}$, define the CAP operators as
follows
\begin{eqnarray*}
a^{+}_{i|n}:&=&P_{n+1}X_{i}P_{n}\\
&=&\sum_{|\beta|=n+1,|\alpha|=n}\frac{1}{\|P^{(a,b)}_{\beta}\|^2\|P^{(a,b)}_{\alpha}\|^2}|P^{(a,b)}_{\beta}\rangle\langle
P^{(a,b)}_{\beta}|
X_i|P^{(a,b)}_{\alpha}\rangle\langle P^{(a,b)}_{\alpha}|\\
&=&\sum_{|\beta|=n+1,|\alpha|=n}\frac{1}{\|P^{(a,b)}_{\beta}\|^2\|P^{(a,b)}_{\alpha}\|^2}\langle
P^{(a,b)}_{\beta},X_{i}P^{(a,b)}_{\alpha}\rangle_{\mu} |P^{(a,b)}_{\beta}\rangle\langle P^{(a,b)}_{\alpha}|\\
a^{0}_{i|n}:&=&P_{n}X_iP_{n}\\&=&\sum_{|\beta|=n,|\alpha|=n}\frac{1}{\|P^{(a,b)}_{\beta}\|^2\|P^{(a,b)}_{\alpha}\|^2}\langle
P^{(a,b)}_{\beta},X_iP^{(a,b)}_{\alpha}\rangle_{\mu}|P^{(a,b)}_{\beta}\rangle\langle P^{(a,b)}_{\alpha}|\\
a^{-}_{i|n}:&=&P_{n-1}X_{i}P_{n}, \ \ \ \ n\geq1\\
&=&\sum_{|\beta|=n-1,|\alpha|=n}\frac{1}{\|P^{(a,b)}_{\beta}\|^2\|P^{(a,b)}_{\alpha}\|^2}\langle
P^{(a,b)}_{\beta},X_iP^{(a,b)}_{\alpha}\rangle_{\mu}|P^{(a,b)}_{\beta}\rangle\langle
P^{(a,b)}_{\alpha}|.
\end{eqnarray*}
From $(\ref{*})$, one has
\begin{eqnarray*}
a^{+}_{i|n}&=&\sum_{|\alpha|=n}\frac{2(\alpha_i+1)(\alpha_i+b_i+a_i+1)}{(2\alpha_i+b_i+a_i+1)(2\alpha_i+b_i+a_i+2)}\frac{1}{\|P^{(a,b)}_{\alpha}\|^2}
|P^{(a,b)}_{\alpha_{0,\dots,0,1,0,\dots,0}}\rangle\langle P^{(a,b)}_{\alpha}|\\
a^{-}_{i|n}&=&\sum_{|\alpha|=n}\frac{2(\alpha_i+a_i)(\alpha_i+b_i)}{(2\alpha_i+b_i+a_i)(2\alpha_i+b_i+a_i+1)}\frac{1}{\|P^{(a,b)}_{\alpha}\|^2}
|P^{(a,b)}_{\alpha_{0,\dots,0,-1,0,\dots,0}}\rangle\langle P^{(a,b)}_{\alpha}|,~n\geq1\\
&& (a^{-}_{i|0}:=0)\\
a^{0}_{i|n}&=&-\sum_{|\alpha|=n}\frac{(a_i^2-b_i^2)}{(2\alpha_i+b_i+a_i)(2\alpha_i+b_i+a_i+2)}\frac{1}{\|P^{(a,b)}_{\alpha}\|^2}
|P^{(a,b)}_{\alpha}\rangle\langle P^{(a,b)}_{\alpha}|
\end{eqnarray*}
Then, for all $\alpha,\beta\in{\N}^d$ such that $|\alpha|=n$ and
$|\beta|=n+1$, one has
\begin{eqnarray*}
a^{+}_{i|n}P^{(a,b)}_{\alpha}&=&\frac{2(\alpha_i+1)(\alpha_i+b_i+a_i+1)}{(2\alpha_i+b_i+a_i+1)(2\alpha_i+b_i+a_i+2)}
P^{(a,b)}_{\alpha_{0,\dots,0,1,0,\dots,0}}\\
a^{-}_{i|n+1}P^{(a,b)}_{\beta}&=&\frac{2(\beta_i+a_i)(\beta_i+b_i)}{(2\beta_i+b_i+a_i)(2\beta_i+b_i+a_i+1)}P^{(a,b)}_{\beta_{0,\dots,0,-1,0,\dots,0}}
\end{eqnarray*}
\begin{equation}{\label{eqJ3}}
a^{0}_{i|n}P^{(a,b)}_{\alpha}=-\frac{(a_i^2-b_i^2)}{(2\alpha_i+b_i+a_i)(2\alpha_i+b_i+a_i+2)}P^{(a,b)}_{\alpha}
\end{equation}
where $-1,1$ are in the $i-th$ index.
\begin{lem}\label{LJ1} For all $1\leq i\leq d,m\in{\N}^*$ and
$\alpha=(\alpha_1,\alpha_2,\dots,\alpha_d)$ such that $|\alpha|=n$,
one has
\begin{equation}\label{eqJ4}
(a^+_i)^mP^{(a,b)}_{\alpha}=\prod_{p=0}^{m-1}\frac{2(\alpha_i+p+1)(\alpha_i+b_i+a_i+p+1)}{(2\alpha_i+2p+b_i+a_i+1)(2\alpha_i+2p+b_i+a_i+2)}
P^{(a,b)}_{\alpha_{0,\dots,0,m,0,\dots,0}}\\
\end{equation} where $$\alpha_{0,\dots,0,m,0,\dots,0}=(\alpha_1,\dots,\alpha_{i-1},\alpha_i+m,\alpha_{i+1},\dots,\alpha_d).$$
\end{lem}
\begin{proof} We prove the above lemma by induction on $m\in{\N}^*$.
\begin{enumerate}
\item[-] For $m=1$, one has $$a^{+}_iP^{(a,b)}_{\alpha}=\frac{2(\alpha_i+1)(\alpha_i+b_i+a_i+1)}{(2\alpha_i+b_i+a_i+1)(2\alpha_i+b_i+a_i+2)}
P^{(a,b)}_{\alpha_{0,\dots,0,1,0,\dots,0}}$$
\item[-] Let $m\geq1$ and suppose that (\ref{eqJ4}) holds true.
Then, one has
\begin{eqnarray*}
(a^+_i)^{m+1}P^{(a,b)}_{\alpha}&=&a^+_i(a^+_i)^mP^{(a,b)}_{\alpha}\\
&=&\prod_{p=0}^{m-1}\frac{2(\alpha_i+p+1)(\alpha_i+b_i+a_i+p+1)}{(2\alpha_i+2p+b_i+a_i+1)(2\alpha_i+2p+b_i+a_i+2)}
a^+_iP^{(a,b)}_{\alpha_{0,\dots,0,m,0,\dots,0}}\\
&=&\prod_{p=0}^{m-1}\frac{2(\alpha_i+p+1)(\alpha_i+b_i+a_i+p+1)}{(2\alpha_i+2p+b_i+a_i+1)(2\alpha_i+2p+b_i+a_i+2)}\\
&&\frac{2(\alpha_i+m+1)(\alpha_i+b_i+a_i+m+1)}{(2\alpha_i+2m+b_i+a_i+1)(2\alpha_i+2m+b_i+a_i+2)}P^{(a,b)}_{\alpha_{0,\dots,0,m+1,0,\dots,0}}\\
&=&\prod_{p=0}^{m}\frac{2(\alpha_i+p+1)(\alpha_i+b_i+a_i+p+1)}{(2\alpha_i+2p+b_i+a_i+1)(2\alpha_i+2p+b_i+a_i+2)}
P^{(a,b)}_{\alpha_{0,\dots,0,m+1,0,\dots,0}}
\end{eqnarray*}
\end{enumerate}
\end{proof}
\begin{theo}
 For all $n\in{\N}$ and $\overline{i}_n=cl\Big(({i}_{1},\dots,{i}_{n})\Big),\overline{j}_n=cl\Big(({j}_{1},\dots,{j}_{n})\Big)\in{\Aa}_n$,
we have
$$\alpha_{e_l|n}e_{\overline{i}_n}=-\frac{(a_l^2-b_l^2)}{(2n_l+b_l+a_l)(2n_l+b_l+a_l+2)}e_{\overline{i}_n}$$ and
the coefficients of $\Omega_n$ in the basis
$\mathcal{B}=(e_{\overline{i}_n})_{\overline{i}_n\in{\Aa}_n}$ are
given by
\begin{eqnarray}\label{eqJ6}
\lambda_{\overline{i}_{n},\overline{j}_{n}}&=&\delta_{\overline{i}_n,\overline{j}_n}\frac{2^{|a|+|b|+d}}{\overline{n}!}
\prod_{i=1}^d\Big(\prod_{p=0}^{n_i-1}\frac{2(p+1)(b_i+a_i+p+1)}{(2p+b_i+a_i+1)(2p+b_i+a_i+2)}\Big)^2 \nonumber \\
&&\prod_{j=1}^d\frac{\Gamma(n_j+a_j+1)\Gamma(n_j+b_j+1)}{(2n_j+b_j+a_j+1)\Gamma(n_j+b_j+a_j+1)}
\end{eqnarray}
where
$$n_{l}=\sharp\Big(\left\{{i}_{k}=l,k=1,\dots,n\right\}\Big),(1\leq l\leq d),\overline{n}=(n_{1},n_{2},\dots,n_{d})$$
with the convention
$$\prod_{p=0}^{-1}\frac{2(p+1)(b_i+a_i+p+1)}{(2p+b_i+a_i+1)(2p+b_i+a_i+2)}=1$$
(this convention is used when $n_i=0$).
\end{theo}
\begin{proof}
Recall that
\begin{eqnarray*}
\lambda_{\overline{i}_{n},\overline{j}_{n}}&=&\langle e_{\overline{i}_{n}},\Omega_{n} e_{\overline{j}_{n}}\rangle_{({\C}^{d})^{\widehat{\otimes}n}}\\
&=& \langle
e_{i_{1}}\widehat{\otimes}e_{i_{2}}\widehat{\otimes}\dots\widehat{\otimes}e_{i_{n}},
\Omega_{n}e_{j_{1}}\widehat{\otimes}e_{j_{2}}\widehat{\otimes}\dots\widehat{\otimes}e_{j_{n}}\rangle_{({\C}^{d})^{\widehat{\otimes}n}}\\
&=&\langle a^{+}_{i_{1}}a^{+}_{i_{2}}\dots a^{+}_{i_{n}}\Phi,a^{+}_{j_{1}}a^{+}_{j_{2}}\dots a^{+}_{j_{n}}\Phi\rangle_\mu\\
&=&\langle(a^{+}_{1})^{m_{1}}(a^{+}_{2})^{m_{2}}\dots(a^{+}_{d})^{m_{d}}\Phi,(a^{+}_{1})^{n_{1}}(a^{+}_{2})^{n_{2}}\dots(a^{+}_{d})^{n_{d}}\Phi\rangle_{\mu}
\end{eqnarray*}
where $\Phi=1_{\Pp}=H_{\widetilde{0}}$, $\widetilde{0}=0_{{\R}^{d}}$
and
$m_{l}=\sharp\left\{i_k=l,~k=1,\dots,n\right\},n_{l}=\sharp\left\{j_k=l,~k=1,\dots,n\right\}$
for all $1\leq l\leq d$. Then, from Lemma \ref{LJ1}, one has
\begin{eqnarray*}
(a^+_{d-1})^{n_{d-1}}(a^+_d)^{n_d}\Phi&=&\prod_{p=0}^{n_d-1}\frac{2(p+1)(b_d+a_d+p+1)}{(2p+b_d+a_d+1)(2p+b_d+a_d+2)}
(a^+_{d-1})^{n_{d-1}}P^{(a,b)}_{\widetilde{0}_{0,\dots,0,n_d}}\\
&=&\prod_{p=0}^{n_d-1}\frac{2(p+1)(b_d+a_d+p+1)}{(2p+b_d+a_d+1)(2p+b_d+a_d+2)}\\
&&\prod_{q=0}^{n_{d-1}-1}\frac{2(q+1)(b_{d-1}+a_{d-1}+q+1)}{(2q+b_{d-1}+a_{d-1}+1)(2q+b_{d-1}+a_{d-1}+2)}P^{(a,b)}_{\widetilde{0}_{0,\dots,0,n_{d-1},n_d}}
\end{eqnarray*}
Repeating the above argument until to obtain
\begin{equation}\label{eqJ5}
(a_1^+)^{n_1}\dots(a_d^+)^{n_d}\Phi=\prod_{i=1}^d\prod_{p=0}^{n_i-1}\frac{2(p+1)(b_i+a_i+p+1)}{(2p+b_i+a_i+1)(2p+b_i+a_i+2)}P^{(a,b)}_{\overline{n}}
\end{equation}
where $\overline{n}=(n_1,n_2,\dots,n_d).$
\begin{enumerate}
\item[(i)] If $\overline{i}_n=\overline{j}_n$, then one has
\begin{eqnarray*}
\lambda_{\overline{i}_{n},\overline{j}_{n}}
&=&\langle(a^{+}_{1})^{m_{1}}(a^{+}_{2})^{m_{2}}\dots(a^{+}_{d})^{m_{d}}\Phi,(a^{+}_{1})^{n_{1}}(a^{+}_{2})^{n_{2}}\dots(a^{+}_{d})^{n_{d}}\Phi\rangle_{\mu}.\\
&=&\prod_{i=1}^d\Big(\prod_{p=0}^{n_i-1}\frac{2(p+1)(b_i+a_i+p+1)}{(2p+b_i+a_i+1)(2p+b_i+a_i+2)}\Big)^2\|P^{(a,b)}_{\overline{n}}\|^2\\
&=&\frac{2^{|a|+|b|+d}}{\overline{n}!}\prod_{i=1}^d\Big(\prod_{p=0}^{n_i-1}\frac{2(p+1)(b_i+a_i+p+1)}{(2p+b_i+a_i+1)(2p+b_i+a_i+2)}\Big)^2\\
&&\prod_{j=1}^d\frac{\Gamma(n_j+a_j+1)\Gamma(n_j+b_j+1)}{(2n_j+b_j+a_j+1)\Gamma(n_j+b_j+a_j+1)}
\end{eqnarray*}
\item[(ii)] If $\overline{i}_n\neq\overline{j}_n$, then
$\left\{{i}_{1},\dots,{i}_{n}\right\}\neq\left\{{j}_{1},\dots,{j}_{n}\right\}$
or there exists $l\in\left\{{i}_{1},\dots,{i}_{n}\right\}$ such that
$m_l\neq n_l$.

\item[-] First case : if
$\left\{{i}_{1},\dots,{i}_{n}\right\}\neq\left\{{j}_{1},\dots,{j}_{n}\right\}$,
then there exists $l\in\left\{1,\dots,d\right\}$ such that
$l\in\left\{{i}_{1},\dots,{i}_{n}\right\}$ and
$l\not\in\left\{{j}_{1},\dots,{j}_{n}\right\}$ or the converse.
Without loss of generality suppose that $l=1$ i.e. $m_{1}\neq0$ and
$n_{1}=0.$ Therefore, one has
\begin{eqnarray*}
\lambda_{\overline{i}_{n},\overline{j}_{n}}&=&\prod_{i=1}^d\prod_{p=0}^{m_i-1}\frac{2(p+1)(b_i+a_i+p+1)}{(2p+b_i+a_i+1)(2p+b_i+a_i+2)}\\
&&\prod_{i=2}^d\prod_{p=0}^{n_i-1}\frac{2(p+1)(b_i+a_i+p+1)}{(2p+b_i+a_i+1)(2p+b_i+a_i+2)}\langle
P^{(a,b)}_{\overline{m}},P^{(a,b)}_{\overline{n}}\rangle\\
&=&0
\end{eqnarray*}
because $ P^{(a,b)}_{\overline{m}}$ and $P^{(a,b)}_{\overline{n}}$
are orthogonal $(\overline{m}\neq\overline{ n}).$

\item[-] Second case : if there exists
$l\in\left\{1,\dots,d\right\}$ such $m_{l}\neq
n_{l}~i.e.~\overline{m}\neq\overline{ n},$ then, one gets
$$\langle P^{(a,b)}_{\overline{m}}, P^{(a,b)}_{\overline{n}}\rangle_{\mu}=0.$$  It follows that
$$\lambda_{\overline{i}_{n},\overline{j}_{n}}=0.$$

Now, let $\overline{i}_n=cl\Big((i_1,i_2,\dots,i_d)\Big)\in{\Aa}_n$.
Recall that
$$U_ne_{\overline{i}_n}:=a^+_{i_1}a^+_{i_2}\dots a^+_{i_d}\Phi.$$
Then, from identities (\ref{eqJ3}) and (\ref{eqJ5}), it follows that
for all $l\in\left\{1,2,\dots,d\right\}$
\begin{eqnarray*}
\alpha_{e_l|n}e_{\overline{i}_n} &: =&U_n^{-1}a^0_{l|n}U_ne_{\overline{i}_n}\\
&=&U^{-1}_na^0_{l|n}a^+_{i_1}a^+_{i_2}\dots a^+_{i_d}\Phi\\
&=&U^{-1}_na^0_{l|n}(a^+_1)^{n_1}\dots(a^+_d)^{n_d}\Phi\\
&=&\prod_{i=1}^d\prod_{p=0}^{n_i-1}\frac{2(p+1)(b_i+a_i+p+1)}{(2p+b_i+a_i+1)(2p+b_i+a_i+2)}U^{-1}_na^0_{l|n}P^{(a,b)}_{\overline{n}}\\
&=-&\prod_{i=1}^d\prod_{p=0}^{n_i-1}\frac{2(p+1)(b_i+a_i+p+1)}{(2p+b_i+a_i+1)(2p+b_i+a_i+2)}\\
&&\frac{(a_l^2-b_l^2)}{(2n_l+b_l+a_l)(2n_l+b_l+a_l+2)}U^{-1}_nP^{(a,b)}_{\overline{n}}\\
&=&-\frac{(a_l^2-b_l^2)}{(2n_l+b_l+a_l)(2n_l+b_l+a_l+2)}U^{-1}_n(a^+_1)^{n_1}\dots(a^+_d)^{n_d}\Phi\\
&=&-\frac{(a_l^2-b_l^2)}{(2n_l+b_l+a_l)(2n_l+b_l+a_l+2)}U^{-1}_na^+_{i_1}a^+_{i_2}\dots a^+_{i_d}\Phi\\
&=&-\frac{(a_l^2-b_l^2)}{(2n_l+b_l+a_l)(2n_l+b_l+a_l+2)}e_{\overline{i}_n}
\end{eqnarray*}
where $n_j=\sharp\left\{i_p=j; \ \ p=1,2,\dots,d\right\},\ \ 1\leq
j\leq d$.
\end{enumerate}
\end{proof}
\subsection{Multiple Gegenbauer polynomials on the cube}
The multiple Gegenbauer polynomials on the cube with parameter
$\lambda=(\lambda_1,\lambda_2,\dots,\lambda_d)$ such that
$\lambda_i>-\frac{1}{2}$ are a particular case of the multiple
Jacobi polynomials with parameter\\
$a=(a_1,a_2,\dots,a_d),~b=(b_1,b_2,\dots,b_d)$ when
$a_i=b_i=\lambda_i-\frac{1}{2}$, $i=1,2,\dots,d$.
\begin{theo}
 For all $n\in{\N}$ and $\overline{i}_n=cl\Big(({i}_{1},\dots,{i}_{n})\Big),\overline{j}_n=cl\Big(({j}_{1},\dots,{j}_{n})\Big)\in{\Aa}_n$,
we have
$$\alpha_{.|n}\equiv0$$ and
the coefficients of $\Omega_n$ in the basis
$\mathcal{B}=(e_{\overline{i}_n})_{\overline{i}_n\in{\Aa}_n}$ are
given by
\begin{eqnarray}\label{eqG1}
\lambda_{\overline{i}_{n},\overline{j}_{n}}&=&\delta_{\overline{i}_n,\overline{j}_n}\frac{2^{2|\lambda|}}{\overline{n}!}
\prod_{i=1}^d\Big(\prod_{p=0}^{n_i-1}\frac{(p+1)(2\lambda_i+p)}{(p+\lambda_i)(2p+2\lambda_i+1)}\Big)^2 \nonumber\\
&&\prod_{j=1}^d\frac{\Big[\Gamma(n_j+\lambda_j+\frac{1}{2})\Big]^2}{(2n_j+2\lambda_j)\Gamma(n_j+2\lambda_j)}
\end{eqnarray}
where
$$n_{l}=\sharp\Big(\left\{{i}_{k}=l,k=1,\dots,n\right\}\Big),(1\leq l\leq d),\overline{n}=(n_{1},n_{2},\dots,n_{d})$$
with the convention
$$\prod_{p=0}^{-1}\frac{(p+1)(2\lambda_i+p)}{(p+\lambda_i)(2p+2\lambda_i+1)}=1$$
(this convention is used when $n_i=0$).
\end{theo}
\begin{proof}
It is sufficient to take
$a_i=b_i=\lambda_i-\frac{1}{2},~~i=1,2,\dots,d$ in (\ref{eqJ6}).
\end{proof}
\subsection{Multiple Chebyshev  polynomials  on the cube}
The multiple Chebyshev  polynomials of first Kind (resp. second
Kind) on the cube are a particular case of the multiple Gegenbauer
polynomials on the cube with parameter
$\lambda=(\lambda_1,\lambda_2,\dots,\lambda_d)$ when $\lambda_i=0~~
(resp.~~\lambda_i=1), \ \ i=1,2,\dots,d$.
\begin{theo}
 For all $n\in{\N}$ and $\overline{i}_n=cl\Big(({i}_{1},\dots,{i}_{n})\Big),\overline{j}_n=cl\Big(({j}_{1},\dots,{j}_{n})\Big)\in{\Aa}_n$,
we have
\begin{enumerate}
\item[i)] If the Jacobi sequences associated to the multiple Chebyshev
polynomials of first kind, then
$$\alpha_{.|n}\equiv0$$ and the coefficients of $\Omega_n$, in the basis
$\mathcal{B}=(e_{\overline{i}_n})_{\overline{i}_n\in{\Aa}_n}$ are
given by
\begin{eqnarray*}
\lambda_{\overline{i}_{n},\overline{j}_{n}}&=&\delta_{\overline{i}_n,\overline{j}_n}\frac{1}{\overline{n}!}
\prod_{i=1}^d\Big(\prod_{p=0}^{n_i-1}\frac{(p+1)}{(2p+1)}\Big)^2
\prod_{j=1}^d\frac{\Big[\Gamma(n_j+\frac{1}{2})\Big]^2}{2n_j\Gamma(n_j)}
\end{eqnarray*}
\item[ii)] If the Jacobi sequences associated to the multiple Chebyshev
polynomials of second kind, then
$$\alpha_{.|n}\equiv0$$ and the coefficients of $\Omega_n$ in the basis
$\mathcal{B}=(e_{\overline{i}_n})_{\overline{i}_n\in{\Aa}_n}$ are
given by
\begin{eqnarray*}
\lambda_{\overline{i}_{n},\overline{j}_{n}}&=&\delta_{\overline{i}_n,\overline{j}_n}\frac{2^{2d}}{\overline{n}!}
\prod_{i=1}^d\Big(\prod_{p=0}^{n_i-1}\frac{(p+2)}{(2p+3)}\Big)^2
\prod_{j=1}^d\frac{\Big[\Gamma(n_j+\frac{3}{2})\Big]^2}{(2n_j+2)\Gamma(n_j+2)}
\end{eqnarray*}
\end{enumerate}
where
$$n_{l}=\sharp\Big(\left\{{i}_{k}=l,k=1,\dots,n\right\}\Big),(1\leq l\leq d),\overline{n}=(n_{1},n_{2},\dots,n_{d})$$
with the convention
$$\prod_{p=0}^{-1}\frac{(p+1)}{(2p+1)}=1~and~\prod_{p=0}^{-1}\frac{(p+2)}{(2p+3)}=1$$
(this convention is used when $n_i=0$).
\end{theo}
\begin{proof} It is sufficient to take $\lambda_i=0$ resp.
$\lambda_i=1,~i=1,2,\dots,d$  in (\ref{eqG1}).
\end{proof}

\subsection{Multiple Legendre polynomials  on the cube}
The multiple Legendre polynomials on the cube are a particular case
of the multiple Gegenbauer polynomials on the cube with parameter
$\lambda=(\lambda_1,\lambda_2,\dots,\lambda_d)$ when\\
$\lambda_i=\frac{1}{2}, \ \ i=1,2,\dots,d$.
\begin{theo}
 For all $n\in{\N}$ and $\overline{i}_n=cl\Big(({i}_{1},\dots,{i}_{n})\Big),\overline{j}_n=cl\Big(({j}_{1},\dots,{j}_{n})\Big)\in{\Aa}_n$,
we have
$$\alpha_{.|n}\equiv0$$ and
the coefficients of $\Omega_n$ in the basis
$\mathcal{B}=(e_{\overline{i}_n})_{\overline{i}_n\in{\Aa}_n}$ are
given by
\begin{eqnarray*}
\lambda_{\overline{i}_{n},\overline{j}_{n}}&=&\delta_{\overline{i}_n,\overline{j}_n}\frac{2^d}{\overline{n}!}
\prod_{i=1}^d\Big(\prod_{p=0}^{n_i-1}\frac{(p+1)^2}{(2p+1)}\Big)^2
\prod_{j=1}^d\frac{\Big[\Gamma(n_j+1)\Big]}{(2n_j+1)}
\end{eqnarray*}
where
$$n_{l}=\sharp\Big(\left\{{i}_{k}=l,k=1,\dots,n\right\}\Big),(1\leq l\leq d),\overline{n}=(n_{1},n_{2},\dots,n_{d})$$
with the convention
$$\prod_{p=0}^{-1}\frac{(p+1)}{(2p+1)}=1$$
(this convention is used when $n_i=0$).
\end{theo}
\begin{proof} It is sufficient to take $\lambda_i=\frac{1}{2}$,
$i=1,2,\dots,d$  in (\ref{eqG1}).
\end{proof}

\end{document}